\documentclass[a4paper,10pt]{amsart}
\usepackage{amsmath}
\usepackage{unicode-math}
\let\mathbbm\mathbb
\usepackage{amsthm}
\usepackage[all]{xy}
\usepackage[CJKbookmarks]{hyperref}
\usepackage{fancyhdr}
\usepackage{tikz-cd}
\usepackage{tikz}

\usepackage{arydshln}

\newtheorem{theorem}{Theorem}[section]
\newtheorem{que}[theorem]{Question}
\newtheorem{lemma}[theorem]{Lemma}
\newtheorem{proposition}[theorem]{Proposition}
\newtheorem{corollary}[theorem]{Corollary}

\theoremstyle{definition}
\newtheorem{definition}[theorem]{Definition}
\newtheorem{example}[theorem]{Example}

\theoremstyle{remark}
\newtheorem{remark}[theorem]{Remark}

\numberwithin{equation}{section}

\DeclareMathOperator{\ad}{ad}

\DeclareMathOperator{\Aut}{Aut}

\DeclareMathOperator{\C}{C}

\DeclareMathOperator{\Der}{Der}

\DeclareMathOperator{\Ext}{Ext}

\DeclareMathOperator{\gr}{gr}

\DeclareMathOperator{\GKdim}{GKdim}

\DeclareMathOperator{\Hom}{Hom}

\DeclareMathOperator{\id}{id}

\DeclareMathOperator{\tr}{tr}

\DeclareMathOperator{\lex}{lex}
\DeclareMathOperator{\wlex}{wlex}

\DeclareMathOperator{\LW}{LW}

\DeclareMathOperator{\UM}{UM}

\makeatletter
\newcommand{\subalign}[1]{%
  \vcenter{%
    \Let@ \restore@math@cr \default@tag
    \baselineskip\fontdimen10 \scriptfont\tw@
    \advance\baselineskip\fontdimen12 \scriptfont\tw@
    \lineskip\thr@@\fontdimen8 \scriptfont\thr@@
    \lineskiplimit\lineskip
    \ialign{\hfil$\m@th\scriptstyle##$&$\m@th\scriptstyle{}##$\hfil\crcr
      #1\crcr
    }%
  }%
}
\makeatother

\begin{document}

\title[Hopf algebras that are not Ore Extensions of enveloping algebras]{Connected Hopf algebras that are not Hopf Ore Extensions of enveloping algebras}

\author{Mengying Hu}
\address{School of Mathematical Sciences, Fudan University, Shanghai 200433, China}
\email{humy21@m.fudan.edu.cn}

\author{Quanshui Wu}
\address{School of Mathematical Sciences, Fudan University, Shanghai 200433, China}
\email{qswu@fudan.edu.cn}

\thanks{This research has been supported by the NSFC (Grant No. 12471032).}


\date{}

\dedicatory{}

\begin{abstract}
We construct a family of connected Hopf algebras with finite Gelfand-Kirillov dimension, none of which is an iterated Hopf Ore extension of the universal enveloping algebra of its primitive part. This provides a negative answer to a question posed by Li and Zhou. It is also demonstrated that these connected Hopf algebras can be formulated as an iterated crossed product of enveloping algebras.  
\end{abstract}

\maketitle

\section{Introduction}
How the coalgebra structure affects the algebra properties of a Hopf algebra is an essential question in the study of Hopf algebras. Suppose $H$ is a Hopf algebra over a field $\mathbb{k}$ with trivial coradical, that is, $H$ is a connected Hopf algebra.  If $H$ is cocommutative, then $H$ is isomorphic to a universal enveloping algebra of a Lie algebra
by Milnor-Moore-Cartier-Kostant Theorem \cite[Theorem
5.6.5]{montgomery1993hopf}. In fact, later Brown and Gilmartin \cite[Question L]{brown2014hopf} asked if any connected Hopf algebra is a universal enveloping algebra of a Lie algebra. It is true when the Hopf algebra is of GK-dimension no larger than $4$ following the classifications \cite{wang2015connected}. Furthermore, all these Hopf algebras are iterated Hopf Ore extensions (IHOE for short) of $\mathbb{k}$. Such kind of results is quite useful for the understanding of the structure of connected Hopf algebras, as the algebraic and homological properties of both enveloping algebras and Ore extensions are rather clear \cite{dixmier1996enveloping,NoetherianMR,liu2014twisted,brown2015connected}. The question was later answered negatively by a counter-example \cite[Theorem 0.5(2)]{BGZ19trans} which is an IHOE but no longer an enveloping algebra of any Lie algebra.
On the other hand, obviously, there are examples that are not IHOEs but enveloping algebras. As a matter of fact, the enveloping algebra of a Lie algebra without a complete flag \cite{brown2015connected} is not an IHOE. These enveloping algebras are the only known examples of connected Hopf algebras that are not IHOEs up to now.  
All known connected Hopf algebras of finite GK-dimension up to now are IHOEs of the enveloping algebras of their primitive parts, including 
\begin{itemize}
    \item the graded ones \cite{zhouRevisited},
    \item the cocommutative ones, which are all enveloping algebras,
    \item the commutative ones, which are polynomial algebras,
    \item all connected Hopf algebras of GK-dimension no more than $4$ \cite{wang2015connected},
    \item the specific example of GK-dimension $5$ in \cite{BGZ19trans}.
\end{itemize} 
So, Li and Zhou \cite[Question C]{zhouRevisited} naturally asked:
\begin{que}\label{conj1}
    Is any connected Hopf algebra $H$ of finite GK-dimension an IHOE of the universal enveloping algebra of its primitive part?
\end{que}
In this paper, we answer this question negatively by constructing a counter-example $\UM(r,2s)$ in Section $3$,  where $r,s$ are non-negative integers such that $r\ge 2s$.

\begin{theorem}[Corollary \ref{CoroMain}]\label{thmintro}
    The connected Hopf algebra $\UM(r,2s)$ is not an IHOE of the universal enveloping algebra of its primitive part. Furthermore, if $r=2s$ and $s\ge 2$, then $\UM(r,2s)$ is not an Ore extension of any Hopf subalgebra of codimension $1$.
\end{theorem}

 As a consequence, there exists a class of connected Hopf algebras that are counter-examples to Question \ref{conj1}. The smallest one among these examples is of GK-dimension 8. Moreover, part of these examples, which are of GK-dimension at least 19, are not IHOEs of any proper Hopf subalgebras.
 
 In Section $5$, we study the structural and homological property of  $\UM(r,2s)$. In  Proposition \ref{PropIterateUL}, we show that
  $\UM(r,2s)$ is an iterated crossed product of enveloping algebras as Hopf algebras. In Proposition \ref{PropCY}, we calculate the Nakayama automorphism of the skew Calabi-Yau algebra $\UM(r,2s)$ by using the filtered deformation technique introduced by Wu and Zhu \cite{wu2021nakayama}. In particular, we show that $\UM(2s,2s)$ is a Calabi-Yau algebra. 

\section{preliminary}
Throughout this paper, all algebras considered are $\mathbbm{k}$-algebras, where  $\mathbbm{k}$ is a field of characteristic $0$. 
A Hopf $\mathbb{k}$-algebra $H$ is called connected if its coradical $H_0=\mathbb{k}$. Such a Hopf algebra can be canonically viewed as a filtered Hopf algebra by its coradical filtration $\{F_m H\}_{m=0}^{\infty}$. The corresponding associated graded Hopf algebra is denoted by $\gr H=\oplus_{m=0}^{\infty} (F_mH/{F_{m-1}}H)$. 
For any $0\ne u\in H$, $u$ is said to be of order $m$ if $u\in F_m H \setminus F_{m-1}H$. For an element $u$ of order $m$, $\overline{u}=u+F_{m-1}\in (\gr H)_m=F_mH/{F_{m-1}}H$ is called the principle symbol of $u$ in literature.

Let $\varepsilon$ be the counit of $H$, $H^+$ be the kernel of $\varepsilon$. For any subspace $K$ of $H$, let $K^+ :=K\cap H^+$. 
Recall that in a connected Hopf algebra $H$, elements of order $1$ in $H^+$ are exactly the primitive elements $P(H)$ and $F_1 H = \mathbb{k} \oplus P(H)$ \cite[Lemma 5.3.2]{montgomery1993hopf}. The set $P(H)$ canonically has a Lie algebra structure.
We refer to  \cite{montgomery1993hopf} as the basic reference for Hopf algebras.

\begin{definition}
 A Hopf algebra $H$ is said to be a Hopf Ore extension (HOE for short) of a Hopf subalgebra $K$ if $H$ is an Ore extension (see \cite[1.9.2]{NoetherianMR}) of $K$ as an algebra. 
 
    A Hopf algebra $H$ is said to be an iterated Hopf Ore extension (IHOE for short) of a Hopf subalgebra $K$ if it contains a chain of Hopf subalgebras:
    \[K=H_{(0)}\subseteq H_{(1)}\subseteq \cdots \subseteq H_{(n)}=H, \]
    such that each $H_{(i)}$ is a Hopf Ore extension of $H_{(i-1)}$. 
    If  $H$ is an IHOE of $K=\mathbb{k}$ then we simply say $H$ is an IHOE.
\end{definition}
If $K$ is a Hopf algebra and $K\subseteq H$ is an Ore extension as an algebra, \cite[Theorem 1.2]{brown2015connected} characterized when $H$ is a Hopf algebra containing $K$ as a Hopf subalgebra. 
For the case that $K$ is connected, see \cite[Theorem 1.3]{brown2015connected}.

Zhuang \cite[Proposition 6.4]{zhuang2013properties} proved that the associated Hopf algebra $\gr H$ of a connected Hopf algebra is commutative (see also \cite[Lemma 5.5]{Andruskiewitsch2000}). If, in addition, $H$ has finite GK-dimension, then more algebra properties are known.
\begin{theorem}\cite[Theorem 6.9]{zhuang2013properties}\label{thmzhuang}
    Let $H$ be a connected Hopf algebra. The following are equivalent.
    \begin{itemize}
        \item[(1)] $\GKdim H  < \infty$;
        \item[(2)] $\GKdim \gr H  < \infty$;
        \item[(3)] $\gr H$ is finitely generated;
        \item[(4)] $\gr H$ is graded isomorphic to a weighted polynomial algebra.
    \end{itemize}
    If one of (1)-(4) holds then $\GKdim H = \GKdim \gr H$. Moreover, $H$ is a noetherian domain of finite global dimension equal to $\GKdim H$.
\end{theorem}
\begin{remark}
    This result is proved in \cite{zhuang2013properties} under the assumption that $\mathbb{k}$ is an algebraically closed field, which is unnecessary as remarked in \cite[Remark 3.8, 3.9]{zhou2020structure}.
\end{remark}
As a consequence, to construct a connected Hopf algebra of finite GK-dimension, from the algebra point of view, a priori, we need to construct a set of weighted generators with lower weighted commutators such that the principle symbols of the generating set is the set of polynomial generators of  $\gr H$.
To specify such kind of generating sets, we focus on the properties of the $\mathbb{k}$-basis lifted from $\gr H$. 

\begin{definition}
    A totally ordered set $\mathcal{G}=\{z_1 < z_2 < \dots < z_n\}$ in a $\mathbb{k}$-algebra $A$ is called a PBW generating set of $A$ if  $\{z_1^{d_1}z_2^{d_2}\cdots z_n^{d_n}\mid d_i\in \mathbb{N}\}$ is a $\mathbb{k}$-basis of $A$.
\end{definition}
The PBW generating set  is named from the associated Poincaré-Birkhoff-Witt-type $\mathbb{k}$-basis, which consists of the lexicographically ordered monomials of $\{z_1 < z_2 < \dots < z_n\}$. 
For convenience, an indexed set $\{z_i\}_{i=1}^n$ is called a PBW generating set if it becomes one with the order $z_1 < z_2 < \dots < z_n$.

In fact, there exists a PBW generating set for any connected Hopf algebra of finite GK-dimension following Theorem \ref{thmzhuang}.
\begin{lemma}\label{lemmaPBWexist}
    Let $H$ be a connected Hopf algebra of finite GK-dimension. Then $H$ has a PBW generating set.
\end{lemma}
\begin{proof}
    Suppose $\GKdim H =n$. By Theorem \ref{thmzhuang}, there exists a set 
    \begin{equation}\label{eqPolyG}
    \mathcal{G}:=\{z_i\}_{i=1}^n\subseteq H \text{ such that } \gr H=\mathbb{k}[\overline{z_1},\cdots,\overline{z_n}].
        \end{equation}
That is to say, the set of the principle symbols of elements in $\mathcal{G}$ is a  polynomial generating set of the graded polynomial algebra $\gr H$.
        
        Let $w_i$ be the order of $z_i$ with respect to the coradical filtration. Then the principle symbol $\overline{z_i}\in (\gr H)_{w_i}$.
    For any $m\in \mathbb{N}$, the space $F_m H$ is spanned by the subset 
    $$S_m:=\{z_1^{d_1}z_2^{d_2}\cdots z_n^{d_n}\mid \sum_i d_i w_i\leqslant m, d_i \in \mathbb{N}\},$$
    following an induction on the filtration. Thus $H$ is spanned by $$S:=\bigcup_{m=0}^{\infty} S_m = \{z_1^{d_1}z_2^{d_2}\cdots z_n^{d_n}\mid d_i\in \mathbb{N}\}.$$ 
    Then $H$ is generated by $\{z_i\}_{i=1}^n$ as an algebra.
    Note that
    \[\dim_{\mathbb{k}} F_m H = \dim_{\mathbb{k}} \oplus_{i=0}^m (\gr H)_i = |\{(d_1, d_2, \cdots, d_n) \mid \sum_{i=1}^n d_i w_i \leqslant m\}|\geqslant |S_m|.\]
    It follows that $S_m$ is a $\mathbb{k}$-linear basis of $F_m H$. Since $m$ is arbitrary and the coradical filtration is exhaustive, $S$ is linearly independent and thus a $\mathbb{k}$-basis of $H$.
\end{proof}
Whenever we consider a PBW generating set $\mathcal{G}$ in a connected Hopf algebra $H$, we assume that $\mathcal{G}\subseteq H^+$. If some element $x\in \mathcal{G}\notin H^+$, just replace it with $x-\varepsilon(x)$, by induction the set is still a PBW generating set. Moreover, if $\mathcal{G}$ satisfies \eqref{eqPolyG}, then the replacement still satisfies it. 

Let $(\mathcal{G},\le)$ be a totally ordered set. The following orders are considered on the words $\langle \mathcal{G}\rangle$ of $\mathcal{G}$.
\begin{definition}\label{lex}
    The lexicographic ordering $\leqslant_{\lex}$ is defined as follows. For any two words $u=z_{i_1}z_{i_2}\cdots z_{i_t}$ and $v=z_{j_1}z_{j_2}\cdots z_{j_{t'}}$, $u\leqslant_{\lex}v$ if and only if either
    \begin{itemize}
        \item $u$ is a prefix of $v$ (that is,  $t\leqslant t'$ and $i_s=j_s$ for all $s\leqslant t$), or
        \item there exists some $s \leqslant \textrm{min}\{t, t'\}$ such that $i_\ell=j_\ell$ for all $\ell<s$ and $i_s<j_s$.
    \end{itemize} 
\end{definition}

\begin{definition}\label{wlex}
    Given a weight map $w:\mathcal{G}\to \mathbb{Z}_+$, each word $u=z_{i_1}z_{i_2}\cdots z_{i_r}\in \langle \mathcal{G}\rangle$ is equipped with a weight defined by $\omega(u)=\sum_{s=1}^t \omega(z_{i_s})$. The corresponding weighted lexicographic ordering $\leqslant_{\wlex}$ is defined as:
    
\quad $u\leqslant_{\wlex} v$ if an only if either
    \begin{itemize}
        \item $\omega(u)< \omega(v)$, or
        \item $\omega(u)=\omega(v)$ and $u\leqslant_{\lex}v$.
    \end{itemize}
\end{definition}
With an order given, one can define the leading word of a polynomial $f\in k\langle\mathcal{G}\rangle$, which is denoted by $\LW(f)$ in this paper.

Let $F=\mathbb{k}\langle\mathcal{G}\rangle$ be the free algebra over the ordered set $\mathcal{G}=\{z_1 < z_2 < \dots < z_n\}$ and $I$ be the ideal generated by $\{z_jz_i - z_iz_j + f_{ij}\mid 1\leqslant i < j\leqslant n\}$, where $\{f_{ij}\}$ are polynomials in $F$. Let $A=F/I$ be the quotient algebra. We sort out a special case of the Diamond Lemma to give an example of $\overline{\mathcal{G}}$ being a PBW generating set of $A$. 
For a quick and nice introduction of the Diamond Lemma and the reduction system we refer to \cite[I.11]{Brown2002LecturesOA}.
For the ordered set $\mathcal{G}$ and the ideal $I$ given above, the corresponding reduction system is $\Re=\{(z_jz_i,z_iz_j - f_{ij})\mid 1\leqslant i < j\leqslant n \}$, where any word $uz_jz_iv$ with $u,v\in \langle\mathcal{G}\rangle$ and $i<j$ can be reduced to $u(z_iz_j-f_{ij})v$. We use the notion $u \xrightarrow{\Re} f$ to denote that the word $u$ can be reduced to $f$  by the reduction system $\Re$.  Reductions can be extended linearly to polynomials in $\mathbb{k}\langle\mathcal{G}\rangle$.

\begin{lemma}\label{PBWlem}
       Let $A=\mathbb{k}\langle \mathcal{G} \rangle/I$, $w: \mathcal{G} \to \mathbb{Z}_+$ be a weight map, and  $\Re$ be the corresponding reduction system given as above. Set $f_{ji}=-f_{ij}$ formally, and $f_{ijk}=[f_{ij},z_k]$.
 Then $\{\overline{z_i}\}_{i=1}^n$ is a PBW generating set of $A$ provided the following 
 conditions hold:
    \begin{itemize}
        \item[(1)] $\omega(z_i) \leqslant \omega(z_j)$, if $i < j$;
        \item[(2)] $\omega(\LW(f_{ij})) < \omega(z_i) + \omega(z_j)$;
        \item[(3)] for all $1 \leqslant i < j < k \leqslant n$,  $f_{ijk}+f_{jki}+f_{kij} \xrightarrow{\Re}0$.
    \end{itemize}
    As a consequence, $A$ is an algebra of GK-dimension $n$ with a PBW basis $$\{\overline{z}_1^{d_1} \overline{z}_2^{d_2}\cdots\overline{z}_n^{d_n}\mid d_i\in \mathbb{N}\}.$$
\end{lemma}
\begin{proof}
    The lemma is a direct corollary of the Diamond Lemma \cite[I.11.6]{Brown2002LecturesOA}. We only need to prove that the weighted lexicographic ordering $\leqslant_{\wlex}$ corresponding to $\omega$ is a semigroup order compatible with $\Re$ satisfying the DCC, and all the ambiguities with respect to the system are resolvable. 
    
    The weighted lexicographic ordering shares all nice properties we need with the length lexicographic ordering (see \cite[I.11.3]{Brown2002LecturesOA}). The conditions (1) and (2) guarantee that $z_i z_j - f_{ij}$ is a linear combination of words less than $z_jz_i$, $\forall \, i<j$. In other words, $\leqslant_{\wlex}$ is compatible with $\Re$. 
    
    There are no inclusion ambiguities in $\Re$ following the definition. For any $i<j<k$, the overlap ambiguity on $z_kz_jz_i$ means that there are two different reductions
    $$z_kz_jz_i \xrightarrow{\Re} (z_jz_k - f_{jk})z_i \textrm{ and } z_kz_jz_i \xrightarrow{\Re} z_k(z_iz_j - f_{ij}).$$  
    It follows from the condition (3) that any two polynomials in $\mathbb{k}\langle\mathcal{G}\rangle$ can be reduced to the same end provided their difference is contained in the ideal generated by $\{[f_{ij},z_k] + [f_{jk},z_i] + [f_{ki},z_j] \mid 1 \leqslant i<j<k \leqslant n\}$,  which shows the last reductions in the following diagram:\\

\begin{tikzcd}[column sep=small]\label{diagramDiamond}
& z_kz_jz_i \arrow[ld,"\Re"'] \arrow[rd,"\Re"] &\\
z_jz_kz_i - f_{jk} z_i \arrow[d,"\Re"']  &   & z_kz_iz_j- z_kf_{ij} \arrow[d,"\Re"]\\
z_jz_iz_k + z_jf_{ki} - f_{jk}z_i \arrow[d,"\Re"'] &  & z_iz_kz_j+ f_{ki}z_j- z_kf_{ij} \arrow[d,"\Re"] \\
z_iz_jz_k - f_{ij}z_k + z_jf_{ki} - f_{jk}z_i \arrow[rr, "\text{reducible difference}",dashed,no head] \arrow[rd, " \Re"', dashed] &   & z_iz_jz_k - z_if_{jk} + f_{ki}z_j- z_kf_{ij} \arrow[ld, "\Re", dashed]\\
&\bullet &
\end{tikzcd}\\
\end{proof}

\section{Umbrella Hopf Algebras}
In this section, we construct a new class of examples of connected Hopf algebras of finite GK-dimension. The idea comes from the efforts to classify connected Hopf algebras generated as an algebra by the second layer of the coradical filtration. We say a connected Hopf algebra $H$ is $m$th-layer generated if it is generated by $F_m H$ as an algebra. Fix a PBW generating set $\mathcal{G}$ as in Lemma \ref{PBWlem}. Note that $\mathcal{G}$ satisfies \eqref{eqPolyG}, which means its set of principle symbols is a polynomial generating set of $\gr H$. As shown in the proof of \cite[Lemma 2.6]{wang2015connected}, we may assume that any element $y$ of order $2$ in $\mathcal{G}$ satisfies that
\begin{align}\label{eqantisym}
    \Delta y \in  y\otimes 1 + 1\otimes y + P(H)\wedge P(H),\end{align}
where the wedge product means the term is an anti-symmetric form. 
Following the notations in \cite{wang2015connected}, we denote by $\delta f :=\Delta f - f\otimes 1 - 1\otimes f$ for any element $f\in H$. We always assume that 
$\mathcal{G}\subseteq H^+.$
For convenience, let $\mathcal{G}_m$ be the subset $\mathcal{G}\cap(F_m H \setminus F_{m-1} H)$.
Then, $\mathcal{G}_1$, which is denoted by $\{x_1, x_2, \cdots, x_r\}$ sometimes, spans the space of all primitive elements $P(H)$. In fact, if $\delta y = \sum_{i,j} \alpha_{ij} x_i \otimes x_j,$ then
it suffices to replace $y$ by $y - \sum_{i,j} \frac{1}{2}\alpha_{ij} x_i x_j$ to meet \eqref{eqantisym}. After the replacement, $\mathcal{G}$ still satisfies \eqref{eqPolyG}, thus $\mathcal{G}$ is still a PBW generating set of $H$ by Lemma \ref{lemmaPBWexist}.

Such an anti-symmetric form is very helpful in the calculations in a $2$nd-layer generated Hopf algebra. For example, one can easily distinguish the elements in the subspace $\mathbb{k}\mathcal{G}_2$ from other elements in $F_2 H$ by its image under $\delta$. Being more precisely,  $\delta(P(H)^2)$ and $\delta(\mathcal{G}_2)$ are symmetric and anti-symmetric tensor forms respectively,   
and $\delta(P(H)^2)\cap \delta(\mathbb{k}\mathcal{G}_2) =0$.
It is worth mentioning that an element $u\in H$ could be determined by $\delta(u)$ up to a difference within $\ker \delta = P(H)$.

We need more effort to distinguish one element $y$ in the generating subset of order $2$ from another. Consider the lexicographic order on $\{x_ix_j\mid 1 \leqslant i,j \leqslant r\}$ as words of length $2$ over $P(H)$. Then $\delta y = \sum_{i>j} \alpha^y_{ij} (x_i\otimes x_j - x_j\otimes x_i)$ has a leading term $x_{i_y}\otimes x_{j_y} - x_{j_y}\otimes x_{i_y}$, where $x_{i_y}x_{j_y}=\max\{x_ix_j\mid \alpha^y_{ij}\ne 0\}$.
In fact, through a certain chosen method of $\mathcal{G}$, the leading term of $\delta(y)$ of $y\in \mathcal{G}_2$ will not occur in the coproduct of any other $y'\in \mathcal{G}_2$. Let us explain this next. Suppose $\mathcal{G}_2=\{y_1,y_2,\cdots ,y_t\}$ and denote $x_{i_{y_p}}x_{j_{y_p}}$ by $x_{i_p}x_{j_p}$ and $\alpha^{y_p}_{ij}$ by $\alpha^{p}_{ij}$. We may assume that
\[x_{i_1}x_{j_1}\leqslant x_{i_2}x_{j_2} \leqslant \cdots \leqslant x_{i_t}x_{j_t}. \]
If $x_{i_\ell}x_{j_\ell}< \cdots < x_{i_t}x_{j_t}$ for some $\ell$,
and $x_{i_k}x_{j_k}= x_{i_\ell}x_{j_\ell}$ for some $k<l$, then replace $y_k$ by $y_k-\frac{\alpha^k_{i_kj_k}}{\alpha^\ell_{i_\ell j_\ell}} y_\ell$.
In finitely many steps, we may assume that the leading terms satisfy
\begin{align}\label{order of leading terms}
x_{i_1}x_{j_1}< x_{i_2}x_{j_2} < \cdots < x_{i_t}x_{j_t}. 
\end{align}

To insure that the leading term $x_{i_k}x_{j_k}$ will not appear in the coproduct of any other $y_l$ with $l\ne k$, next we do similar operation as above. 

Obviously, $\alpha^l_{i_1j_1}=\alpha^l_{i_2j_2}=\cdots=\alpha^l_{i_{l-1}j_{l-1}}=0$ for any $l$. 
If $\alpha^k_{i_lj_l}\ne 0$ for some $k>l$, then by replacing $y_k$ by $y_k-\frac{\alpha^k_{i_lj_l}}{\alpha^l_{i_lj_l}} y_l$,  we may have  $\alpha^k_{i_lj_l}=0$. By induction, we may assume that $\alpha^k_{i_lj_l}=0$ for any $k\ne l$.
In other words, the pair $(x_{i_y},x_{j_y})$ would be the unique feature of $y$. By the feature, it is relatively easy to decide the coefficients of $y$ as an element of  the PBW basis for any element in $F_2 H$. Moreover, the coproducts of generators in higher layers of the coradical filtration have close relation with the commutativity of the elements in $F_1 H$ in turn, see \cite[Proposition 4.17]{wang2015connected}
as a detailed example, where the coproduct of generators in $\mathcal{G}_3$ forces $P(H)$ to be a trivial Lie algebra. In the $2$nd-layer generated cases of our concern, if we collect the set of these feature pairs together as
\begin{equation}\label{equLambda}
    \Lambda = \{(x_{i_y},x_{j_y})\mid y\in \mathcal{G}_2\},
\end{equation}
the larger the set is, the more commutative $P(H)$ as a Lie algebra is. The following example, with  $|\mathcal{G}_1|=3$, provides an intuition for such property. 
\begin{example}
    Suppose $H$ is $2$-layer generated and $P(H)=\sum_{i=1}^3 \mathbb{k} x_i$. Since $P(H)\wedge P(H)$ is of linear dimension $3$, $|\mathcal{G}_2|\leqslant 3$.
    
    \bullet Case $\mathcal{G}_2=\emptyset$. $P(H)$ can be any Lie algebra of dimension $3$.

    \bullet Case  $\mathcal{G}_2 = \{y\}$. Then $|\Lambda|=1$, and  \cite[Theorem 3.5]{wang2015connected} listed all possible structures of $H$. Among which, all but case (b) have non-trivial $P(H)$. Actually, the possible Lie algebra structures of these primitive parts include all Lie algebras of dimension $3$ but $\mathfrak{sl}_2$.

    \bullet Case $|\mathcal{G}_2|=2$. Then $|\Lambda|=2$. By some tedious but basic calculations, $P(H)$ is either a trivial Lie algebra or isomorphic to the one with:
    \[[x_1,x_2] = 2x_2,[x_1,x_3]=x_3,[x_2,x_3]=0.\]

    \bullet Case $|\mathcal{G}_2|=3$. Then $|\Lambda| = 3$, and  $P(H)$ must be a trivial Lie algebra.
\end{example}
Obviously, if $P(H)$ is commutative, then, by \cite[Lemma 2.5]{wang2015connected}, $H$ is an enveloping algebra. 

The following lemma shows that such commutativity affects to what extent of the commutativity of the higher layers in the coradical filtration.
\begin{proposition}\label{commutator-corad-fil}

   Let $H$ be a connected Hopf algebra. For a fixed positive integer $k$, the following are equivalent.
    \begin{enumerate}
       \item $[F_m H,F_l H]=0$ for any $m+l\leqslant k$ with $m,l\geqslant 0$.
        \item $[F_m H,F_l  H]\subseteq F_{m+l-k} H$ for any integer $m,l\geqslant 0$.
    \end{enumerate}
\end{proposition} 
\begin{proof} 
  Since $H$ is connected, $F_0 H\cap H^+ =0$. It follows from $[H,H]=[H^+,H^+]\subseteq H^+$ that $F_0 H\cap [H,H]=0$. So, obviously, (2) implies (1). 
  
  Suppose (1) holds, that is, $[F_m H, F_l  H]=0\subseteq F_{m+l-k} H$ when $m+l\leqslant k$.  We prove (2) by induction. Suppose, for integer $ N\geqslant k$, $[F_m H, F_l  H]\subseteq F_{m+l-k} H$   whenever $m+l\leqslant N$. We prove that $[u,v]\in F_{N+1-k}H$ for any element $u\in F_m H$ and $v\in F_l  H$, where $m+l=N+1$. By the basic property of the coradical filtration \cite[5.2.2]{montgomery1993hopf}, $\delta(u)$, denoted by $u_1{}'\otimes u_2{}'$ below, lies in $\sum_{i=1}^{m-1}F_i H \otimes F_{m-i} H$. Similarly, $\delta(v)=v_1{}'\otimes v_2{}'\in \sum_{j=1}^{l-1}F_j H \otimes F_{l-j} H$. Thus, by a direct computation,
    \begin{align*}
        \Delta [u,v] =& [u,v]\otimes 1 + 1\otimes [u,v] \, + \\
        & [u,v_1{}']\otimes v_2{}' + v_1{}'\otimes [u,v_2{}']+[u_1{}',v]\otimes u_2{}' \, + u_1{}'\otimes [u_2{}',v] \, + \\
        &[u_1{}',v_1{}']\otimes u_2{}'v_2{}' + v_1{}'u_1{}'\otimes[u_2{}',v_2{}'].
        \end{align*} 
      Hence
        \begin{align*}
        \delta [u,v] = & \Delta [u,v]  -  [u,v]\otimes 1 - 1\otimes [u,v] \, \in \\
& \sum_{j=1}^{l-1}( [F_{m}H,F_j H] \otimes F_{l-j}H +F_j H  \otimes [F_m H,F_{l-j} H]) \, + \\
       &\sum_{i=1}^{m-1} ([F_i H,F_l  H] \otimes F_{m-i}H +F_i H  \otimes [F_{m-i} H,F_l  H]) \, + \\
        &\sum_{\substack{1\leqslant i\leqslant m-1\\ 1\leqslant j\leqslant l-1}} ([F_i H,F_j H]\otimes F_{m+l-i-j}H + F_{i+j}H \otimes [F_{m-i} H,F_{l-j} H]).
    \end{align*}
     It follows from the induction hypothesis that
    \begin{align*}  
   \delta [u,v] & \in \\
& \sum_{j=k+1-m}^{l-1} F_{m+j-k}H \otimes F_{l-j}H + \sum_{j=1}^{N-k} F_j H  \otimes F_{N+1-j-k}H \, + \\
       &\sum_{i=k+1-l}^{m-1} F_{i+l-k} H \otimes F_{m-i}H + \sum_{i=1}^{N-k} F_i H  \otimes F_{N+1-i-k}H \, + \\
        &\sum_{(i+j)=k+1}^{N-1} F_{i+j-k}H\otimes F_{N+1-i-j}H + \sum_{(i+j)=2}^{N-k} F_{i+j}H \otimes F_{N+1-i-j-k}  H.
    \end{align*}
    Therefore, $\delta [u,v]\in \sum_{i=1}^{N-k} F_{i}H\otimes F_{N+1-i-k}H$. It follows from the  definition of the coradical filtration that $[u,v]\in F_{N+1-k}H$. Hence (2) holds.
\end{proof}

The following is \cite[Proposition 6.4]{zhuang2013properties}, which is a corollary of Proposition \ref{commutator-corad-fil}.

\begin{corollary}
    Let $H$ be a connected Hopf algebra. Then the associated graded $\gr H$ of $H$ with respect to the coradical filtration is commutative.
\end{corollary}
\begin{proof}
   It follows from $[F_0 H,F_1 H]=0$ and Proposition \ref{commutator-corad-fil} where $k=1$.
\end{proof}

If the Hopf algebra $H$ itself is non-commutative, then there is a maximal integer $k$ satisfying the conditions (1) or (2) in Proposition \ref{commutator-corad-fil}. Following \cite[(5.2)]{wu2021nakayama}, the commutative graded algebra $\gr H$ is a Poisson algebra with a non-trivial bracket
\[\{\overline{u},\overline{v}\}:=[u,v] + F_{m+l-k-1} H, \overline{u}\in (\gr H)_m,\overline{v}\in (\gr H)_l.\]
Later in \S $5$ we will see that the method developed in \cite{wu2021nakayama} enables us to calculate the Nakayama automorphism of the example.

 If the Lie algebra $P(H)$ is trivial, then $[F_1 H,F_1 H] = 0$. As a consequence of Proposition \ref{commutator-corad-fil}, $[F_2 H,F_1 H]\subseteq F_1 H$. For any generator $y\in \mathcal{G}_2$, it is direct to check that the subalgebra $K$ generated by $P(H)\cup \{y\}$ has $\mathcal{G}_1\cup \{y\}$ as a PBW generating set. Thus, $K$ would be an HOE of the primitive part, which advances the process for $H$ to be an IHOE of $U(P(H))$. The complexity of finding counter-examples will be greatly increased in such cases. 
 So, we, on the other hand, consider the situation where $P(H)$ decomposes into the direct sum of two Lie subalgebras such that one of them never occurs in the coproduct of any generator in $\mathcal{G}_2$ so that it could be relatively free in the algebra structure. Denote this Lie subalgebra of $P(H)$ by $\mathfrak{s}$ and its supplement by $\mathfrak{L}=\mathbb{k}x_1 \oplus \mathbb{k}x_2 \oplus \cdots \oplus \mathbb{k}x_r$. 
 For convenience we also denote $\mathcal{G}_2  =\{y_i\}_{i=1}^t$. Suppose that $\Delta(y_i)= y_i\otimes 1 + 1\otimes y_i + \sum_{j<k}\alpha_{j k}(x_j\otimes x_k - x_k\otimes x_j)$. Then its commutator with an element $M\in \mathfrak{s}$ satisfies:
\begin{align}
    \Delta[y_i,M] - [y_i,M]\otimes 1 - 1\otimes [y_i,M] = \sum_{j<k}\alpha_{j k} \big([x_j,M]\otimes x_k -x_j\otimes [x_k,M]\big).\label{equationMypre}
\end{align}
Now that $\mathfrak{s}$ will not occur in an anti-symmetric form in the coproduct of any element in $F_2H$, it is natural to make the further assumption that $[\mathfrak{s}, \mathfrak{L}]\subseteq \mathfrak{L}$. 
As a consequence, $\mathfrak{L}$ is a Lie module of $\mathfrak{s}$. The property of $\mathfrak{L}$ as Lie $\mathfrak{s}$-module is of our main interest, as it also deeply affects $\mathcal{G}_2$ by \eqref{equationMypre}. The perspective is further focused to the case that $\mathfrak{s}\hookrightarrow \mathfrak{gl}_r(\mathbb{k})$.

Now, we introduce the following construction of an algebra to give an answer to Question \ref{conj1}.
\begin{example}[$\UM(A)$]\label{example}
    Let $A\in M_{r}(\mathbb{k})$ be a fixed anti-symmetric matrix. Then $\mathfrak{so}(A):= \{M\in M_{r}(\mathbb{k}) \mid MA=-AM^T\}$ is a Lie subalgebra of $\mathfrak{gl}_r(\mathbb{k})$. 
    In fact, $\mathfrak{so}(A) = \{M\in M_{r}(\mathbb{k}) \mid MA=(MA)^T\}$.
    Suppose $b(\mathfrak{so}(A))$ is a $\mathbb{k}$-basis of $\mathfrak{so}(A)$. Let $\UM(A)$ be the algebra generated by the set $\mathcal{G}=\{x_i\}_{i=0}^{r}\cup \{y_i\}_{i=1}^{r}\cup b(\mathfrak{so}(A))$ subject to the relations \eqref{relation1}-\eqref{relationlast}. The Lie bracket of $\mathfrak{so}(A)$ is denoted by $\{-,-\}$ to distinguish from the commutator $[-,-]$ in $\mathbb{k}\langle\mathcal{G}\rangle$. The notation $M_{ij}$ is used to denote the element in the $(i, j)$-position in a matrix  $M$.
\begin{align}
    [x_0,-] &, ~ & \label{relation1}\\
    [x_i,x_j] &- A_{ij} x_0,& 1\leqslant i,j \leqslant r,\label{relationxx}\\
    [y_i,y_j] &- \frac{1}{3} A_{ij} x_0^3,& 1\leqslant i,j \leqslant r,\label{relationyy}\\
    [x_i,y_j] &, & 1\leqslant i, j\leqslant r,\label{relationxy}\\
    [x_i,M] &- \sum_{k=1}^{r} M_{ik} x_k,& M\in b(\mathfrak{so}(A)),\ 1\leqslant i \leqslant r, \label{relationMx}\\
    [y_i,M] &- \sum_{k=1}^{r} M_{ik} y_{k}, & M\in b(\mathfrak{so}(A)),\ 1\leqslant i \leqslant r, \label{relationMy} \\
    [M,N] & - \sum_{X \in  b(\mathfrak{so}(A))} \alpha_XX,  & \textrm{ if } \{M, N\} = \sum_{X \in  b(\mathfrak{so}(A))} \alpha_XX, \,M,N\in b(\mathfrak{so}(A)).\label{relationlast} 
\end{align}
Suppose $b(\mathfrak{so}(A))=\{X_i\}_{i=1}^d$, where $d$ is the linear dimension of $\mathfrak{so}(A)$. We fix a total order on $\mathcal{G}$ as
$$x_0<x_1< \dots <x_{r}<X_1< \dots < X_d <y_{1}<y_{2}< \dots <y_{r}.$$  
Consider the weighted lexicographic order on the words $\langle\mathcal{G}\rangle$ with 
\begin{align} \label{wlexH2so}
   w(x_i) = w(X_j)=1, \ 0\leqslant i\leqslant r, 1\leqslant j\leqslant d; \  w(y_i)=2, \  1\leqslant i\leqslant r.
\end{align}
Then the weight together with relations \eqref{relation1}-\eqref{relationlast} satisfy the conditions (1) and (2) in Lemma \ref{PBWlem}.

The reduction system corresponding to relations (\ref{relation1})-(\ref{relationlast}) is denoted as $\Re$. For any generators $z_i<z_j\in \mathcal{G}$, the reduction is denoted as $z_jz_i \xrightarrow{\Re} z_iz_j - f_{z_iz_j}$. For example, 
$$M x_i \xrightarrow{\Re} x_i M - f_{x_iM}, \textrm{ where } f_{x_iM}= \sum_k M_{ik}x_k, \, M\in b(\mathfrak{so}(A)), \ 1\leqslant i \leqslant r.$$

We check the condition (3) in Lemma \ref{PBWlem} for all possible triples of the generators in the following.
\begin{itemize}
\item Case $(x_0,z_i,z_j)$ where $\{z_i, z_j\}\subseteq \mathcal{G}$. 

Obviously, any monomial $u x_0$ can be reduced to $x_0 u$ in finite steps. Then
\begin{align*}
    ~& [f_{x_0z_i},z_j] + [f_{z_iz_j},x_0] + [f_{z_ix_0},z_i]\\
    = & [0, z_j]+f_{z_iz_j}x_0 - x_0f_{z_iz_j} + [0,z_i]\\
\xrightarrow{\Re} & x_0f_{z_iz_j}-x_0f_{z_iz_j}=0.
\end{align*}
\item Case $(z_i,z_j,z_k)$ where $\{z_i,z_j,z_k\}\subseteq \{x_i\}_{i=1}^r\cup \{y_i\}_{i=1}^r$. 

Trivial as  $f_{z_iz_j},f_{z_jz_k},f_{z_kz_i} \in \mathbb{k}[x_0]$.

\item Case $(L,M,N)$  where $L,M,N\in b(\mathfrak{so}(A))$. 

Note that $f_{LM}$ is a linear combination of the elements in $b(\mathfrak{so}(A))$, and so is $[f_{LM},N]$. The condition (3) follows from the Jacobi identity of the Lie algebra $\mathfrak{so}(A)$.
\item Case $(x_i,M,N)$ where $i\geqslant 1$ and $M,N\in b(\mathfrak{so}(A))$.  

By relation \eqref{relationlast}, $f_{MN}=\sum_{X\in b(\mathfrak{so}(A))} \alpha_X X\in M_r(\mathbb{k})$ is the commutator of matrices $M$ and $N$. Then
\begin{align*}
    & [f_{x_iM},N] + [f_{MN},x_i] + [f_{Nx_i},M]\\
    =& [\sum_j M_{ij}x_j,N] + \sum_{X\in b(\mathfrak{so}(A))}\alpha_X [X,x_i] + [\sum_j  -N_{ij}x_j,M]\\
    \xrightarrow{\Re} & \sum_j M_{ij}(x_jN -(x_jN -\sum_k N_{jk}x_k)) +\\
   \hspace{4em}& \sum_X \alpha_X ((x_iX - \sum_j X_{ij}x_j) - x_iX) +
   \sum_j -N_{ij}(x_jM-(x_jM-\sum_k M_{jk}x_k))\\
  =& \sum_{j,k}M_{ij}N_{jk}x_k - \sum_j (\sum_X \alpha_X X_{ij}) x_j  - \sum_{j,k}N_{ij}M_{jk} x_k\\
  =& \sum_k -\{M, N\}_{ik} x_k +\sum_k(MN-NM)_{ik}x_k
=0. 
\end{align*}
Actually the $r$-dimensional space $\bigoplus_{i=1}^r \mathbb{k} x_i$ is the canonical $r$-dimensional right $\mathfrak{so}(A)$-Lie module, which ensures the reduction.

\item Case $(M,N,y_i)$ where $i\geqslant 1$ and $M,N\in b(\mathfrak{so}(A))$. 

Similar to the case $(x_i,M,N)$.
\item Case $(x_i,M,y_j)$ where $1 \leqslant i, j \leqslant r,M\in b(\mathfrak{so}(A))$. 

Trivial by the fact that any $y_j x_i$ can be reduced to $x_i y_j$.

\item Case $(x_i,x_j, M)$  where $1 \leqslant i, j \leqslant r$ and $M\in b(\mathfrak{so}(A))$. 
\begin{align*}
    \hspace{4em}~&[f_{Mx_i},x_j]  + [f_{x_jM},x_i] + [f_{x_ix_j},M]\\
     \hspace{4em}= & -\sum_k M_{ik}[x_k,x_j]  + \sum_k M_{jk}[x_k,x_i] + A_{ij}[x_0,M]\\
     \hspace{4em}\xrightarrow{\Re}& -\sum_{k<j} M_{ik} (x_kx_j -(x_kx_j -A_{kj}x_0))- \sum_{k>j}M_{ik}(x_jx_k - A_{jk}x_0-x_jx_k) \ +\\
     \hspace{4em}~&\ \sum_{k<i} M_{jk}(x_kx_i -(x_kx_i-A_{ki}x_0) ) + \sum_{k>i}M_{jk}(x_ix_k - A_{ik}x_0-x_ix_k) + 0\\
     \hspace{4em}=&-\sum_k M_{ik}A_{kj} x_0 + \sum_k M_{jk}A_{ki} x_0\\
     \hspace{4em}=&-(MA)_{ij}x_0 + (MA)_{ji}x_0 = 0,
     \end{align*}
     where the second last equality follows from the anti-symmetry of $A$, and the last equality follows from that $MA$ is a symmetric matrix for any $M\in\mathfrak{so}(A)$. 

\item Case $(M,y_i,y_j)$ where $1 \leqslant i, j \leqslant r$ and $M\in b(\mathfrak{so}(A))$. 

This is similar to the case $(x_i,x_j, M)$.
\end{itemize}

\begin{corollary}
 $\GKdim(\UM(A))=2r+1+\dim_{\mathbb{k}}\mathfrak{so}(A).$  
\end{corollary}
\begin{proof}
  It follows from Lemma \ref{PBWlem} that $\mathcal{G}$ is a PBW generating set of the algebra $\UM(A)$.  
\end{proof}
\end{example}

\begin{proposition}
    The algebra $\UM(A)$ given in Example \ref{example} can be endowed with a Hopf algebra structure as in the following.    
    \begin{align}
        &\{x_i\}_{i=0}^r \cup \mathfrak{so}(A) \subseteq P(\UM(A)), \label{x-M-comultiplication}\\
        &\Delta y_i = y_i\otimes 1 + 1\otimes y_i + x_0 \otimes x_i - x_i\otimes x_0, \, \varepsilon (y_i)=0, 1\leqslant i\leqslant r, \label{y-comultiplication}\\
        &S(y_i) = -y_i , 1\leqslant i\leqslant r \label{y-antipode}.
    \end{align}
    By such coproduct $H=\UM(A)$ is a connected Hopf algebra.
\end{proposition}
\begin{proof}
    By definition, $\UM(A)$ is the quotient algebra of $\mathbb{k}\langle\mathcal{G}\rangle$ subject to the relations \eqref{relation1}-\eqref{relationlast}. It is routine to check that \eqref{x-M-comultiplication}, \eqref{y-comultiplication} and \eqref{y-antipode} give a Hopf algebra structure on the free algebra $\mathbb{k}\langle\mathcal{G}\rangle$. Let $I$ be the ideal generated by the relations \eqref{relation1}-\eqref{relationlast}. Then $S(I)\subseteq I$. 
    Note that the commutator of two primitive elements is still primitive. For other generators of $I$, 
\begin{align*}
   \Delta([x_0,u])=& [x_0,u_1]\otimes u_2 + u_1\otimes[x_0,u_2],\forall  u \in \mathbb{k}\langle\mathcal{G}\rangle \textrm{ with } \Delta(u)=u_1 \otimes u_2,\\
    \Delta([x_i,y_j])=&[\Delta x_i,\Delta y_j]\\
    =&[x_i,y_j]\otimes 1 + 1\otimes[x_i,y_j] + x_0\otimes [x_i,x_j]-[x_i,x_j]\otimes x_0\\
    =& [x_i,y_j]  \otimes 1 + 1\otimes[x_i,y_j] + \\ &x_0\otimes ([x_i,x_j]-A_{ij}x_0)-([x_i,x_j]-A_{ij}x_0)\otimes x_0,\\
        \Delta([y_i,y_j]-&\frac{1}{3}A_{ij}x_0^3)=[\Delta y_i,\Delta y_j]-\frac{1}{3}(\Delta x_0)^3\\
    =&[y_i,y_j]\otimes 1 + 1\otimes [y_i,y_j] + \\ &x_0^2\otimes[x_i,x_j] + [x_i,x_j]\otimes x_0^2 - \frac{1}{3}A_{ij}(x_0\otimes 1 + 1\otimes x_0)^3\\
    =&([y_i,y_j]-\frac{1}{3}A_{ij}x_0^3)\otimes 1 + 1 \otimes ([y_i,y_j]-\frac{1}{3}A_{ij}x_0^3).
\end{align*}
    \begin{align*}
    \Delta([y_i,M]-&\sum_{k=1}^r M_{ik} y_k) = [\Delta y_i,\Delta M] - \sum_{k=1}^r M_{ik}\Delta y_k \\
    =&[y_i,M]\otimes 1 + 1\otimes [y_i,M] + x_0\otimes [x_i,M] - [x_i,M]\otimes x_0 -\\ &\sum_{k=1}^rM_{ik}(y_k\otimes 1 + 1\otimes y_k + x_0\otimes x_k - x_k\otimes x_0)  \\
    =&([y_i,M] - \sum_{k=1}^rM_{ik}y_k)\otimes 1 + 1\otimes ([y_i,M] - \sum_{k=1}^rM_{ik}y_k) +\\ &x_0\otimes([x_i,M]-\sum_{k=1}^r M_{ik}x_k) - ([x_i,M]-\sum_{k=1}^r M_{ik} x_k)\otimes x_0.\\
    \end{align*}
It follows that $\Delta(I)\subseteq I\otimes \mathbb{k}\langle\mathcal{G}\rangle + \mathbb{k}\langle\mathcal{G}\rangle \otimes I$. Hence $I$ is a Hopf ideal of $\mathbb{k}\langle\mathcal{G}\rangle$.
  It follows that \eqref{x-M-comultiplication}, \eqref{y-comultiplication} and \eqref{y-antipode} induce a Hopf algebra structure on $\UM(A)$.

  Obviously by setting $\tilde{F}_0 H=\mathbb{k}$ and $\tilde{F}_{i+1} H =\Delta^{-1}(H\otimes \tilde{F}_i H + \tilde{F}_0 H \otimes H)$, $\bigcup_{i=0}^{\infty} \tilde{F}_i H$ is an exhaustive coalgebra filtration of $H$. By Lemma \cite[5.3.4]{montgomery1993hopf} the coradical $F_0 H \subseteq \tilde{F}_0 H$, thus $H$ is a connected Hopf algebra.
\end{proof}

Recall that the structure matrix $A$ of $\UM(A)$ is an anti-symmetric matrix. Suppose $A\in M_r(\mathbb{k})$ and rank$(A) = 2s$. Then $A$ has an orthogonal normalized form
\begin{align}\label{matrixDia}
    B=\begin{bmatrix}
    0 & 1 & && &&&\\
    -1 & 0 & && &&&\\
    && \ddots && &&&\\
    &&& 0 & 1 &&&\\
    &&& -1&0 & &&\\
    &&&&&0&&\\
    &&&&&& \ddots& \\
    &&&&&&&0 \\
\end{bmatrix} 
\begin{array}{c}
     \overline{\uparrow}\\  \\ 2s \\  \\ \underline{\downarrow} \\ \overline{\uparrow} \\ r-2s \\ \underline{\downarrow} \\ 
\end{array}.
\end{align}
The following lemma tells us that  $\UM(A)$ actually can be given by its normalized matrix.
\begin{proposition}
 If the matrix $A$ is congruent to another anti-symmetric matrix $A'$, then $\UM(A)\cong \UM(A')$ as Hopf algebras. Specifically, let $B$ be the orthogonal normalization of $A$. Then $\UM(A)\cong \UM(B)$ as Hopf algebras.
\end{proposition}
\begin{proof}
By assumption, there exists an orthogonal matrix $P$ such that $PAP^T=A'$. Then it is obvious that for any matrix $M\in M_r(\mathbb{k})$, $M\in \mathfrak{so}(A)$ if and only if $PMP^T\in \mathfrak{so}(A')$.
    Consider the set $\mathcal{G}'=\{x_i'\}_{i=0}^r \cup b(\mathfrak{so}(A')) \cup \{y_i'\}_{i=1}^r$ in $\UM(A')$, where $b(\mathfrak{so}(A'))= \{PMP^T\mid M\in b(\mathfrak{so}(A))\}$ is a basis of $\mathfrak{so}(A')$. The algebra $\UM(A')$ is generated by $\mathcal{G}'$ subject to relations parallel to $\eqref{relation1}$-$\eqref{relationlast}$. Consider the following map between the two algebras:
    \begin{align*}
        \Phi:\UM(A)&\to \UM(A'),\\
        x_0 &\mapsto x_0',\\
        x_i &\mapsto \sum_j P_{ji}x_j',\\
        y_i &\mapsto \sum_j P_{ji}y_j',\\
        M &\mapsto PMP^T.\label{mapM}
    \end{align*}
    In fact, the map gives an algebra isomorphism between the free algebras $\mathbb{k}\langle \mathcal{G}\rangle$ and $\mathbb{k}\langle \mathcal{G}'\rangle$. Moreover, it maps exactly the ideal defined by \eqref{relation1}-\eqref{relationlast} in $\mathbb{k}\langle \mathcal{G}\rangle$ to that in $\mathbb{k}\langle \mathcal{G}'\rangle$. For example, for the relation \eqref{relationMx},
    \begin{align*}
        & [x_i, M] - \sum_{k=1}^r M_{ik} x_k  \\
     \mapsto &  
      \sum_{j=1}^r [P_{ji} x'_j,PMP^T]- \sum_{k=1}^r M_{ik}  (\sum_{l=1}^r P_{lk} x'_l)\\ 
     &= \sum_{j=1}^r P_{ji} [x'_j,PMP^T] - \sum_{l=1}^r (MP^T)_{il} x'_l\\
     &= \sum_{j=1}^r P_{ji} [x'_j,PMP^T] -\sum_{l=1}^r (P^T PMP^T)_{il} x'_l\\
     &= \sum_{j=1}^r P_{ji} [x'_j,PMP^T] - \sum_{l=1}^r \sum_{j=1}^r (P^T)_{ij}(PMP^T)_{jl} x'_l\\
       & = \sum_{j=1}^r P_{ji}\left([x'_j, PMP^T] - \sum_{l=1}^r (PMP^T)_{jl} x'_l \right). 
    \end{align*} 
    It follows that $\Phi:\UM(A) \to \UM(A')$ is an algebra isomorphism.
    It is routine to check that $\Phi$ is a Hopf algebra homomorphism.
\end{proof}
The Hopf algebra $\UM(B)$ defined above is named as the \textit{Umbrella Hopf algebra} for its structural properties. Recall that in a $2$nd-layer generated connected Hopf algebra $H$ with a fixed generating set $\mathcal{G}$, we define a set $\Lambda$ to characterize its coproduct of the generators in $\mathcal{G}_2$. Figures \ref{figoc} and \ref{figUm} below show, from the top and bottom view respectively, an undirected graph corresponding to $\mathcal{G}$, whose vertices are $\mathcal{G}_1$, and there is an edge between two vertices if they appear as a pair in $\Lambda = \{(x_i,x_0)\mid 1\leqslant i \leqslant r \}$ suggested by \eqref{y-comultiplication}. These essential parts are drawn in bullets and bold lines.  
We use the notation $\UM(r,2s)$ to denote $\UM(B)$ for $B\in M_r(\mathbb{k})$ with rank$(B)=2s$. As is shown in Figure \ref{figoc}, $r$ is the number of ribs and $s$ is the number of gray pieces of canopy. The relations \eqref{relationxx} and \eqref{relationyy} in $\UM(B)$ are commutators except for
\begin{align}
    [x_{2i-1},x_{2i}]&-x_0,& 1\leqslant i \leqslant s,\tag{3.6'}\\
    [y_{2i-1},y_{2i}]&-\frac{1}{3} x_0^3,& 1\leqslant i \leqslant s. \tag{3.7'}
\end{align}
Thus on the whole canopy, a non-trivial commutator always appears between elements around a gray piece, which is marked by gray solid line. The generators in $\mathfrak{so}(A)$ could be viewed as the umbrella shaft as in Figure \ref{figUm}, which may have non-trivial commutators with both $\{x_i\}$ and $\{y_j\}$. 
\begin{figure}[htbp]
    \centering
    \begin{minipage}{0.48\textwidth}
        \centering
        \includegraphics[width=\linewidth]{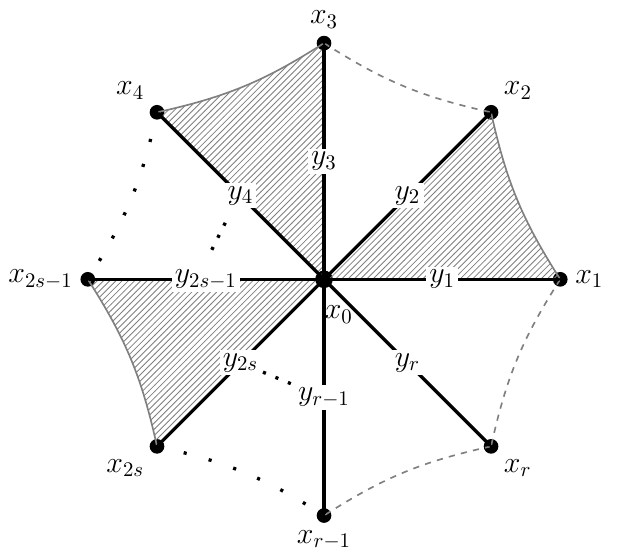} \caption{Top view}\label{figoc}
    \end{minipage}%
    \hfill 
    \begin{minipage}{0.48\textwidth}
        \centering
        \includegraphics[width=\linewidth]{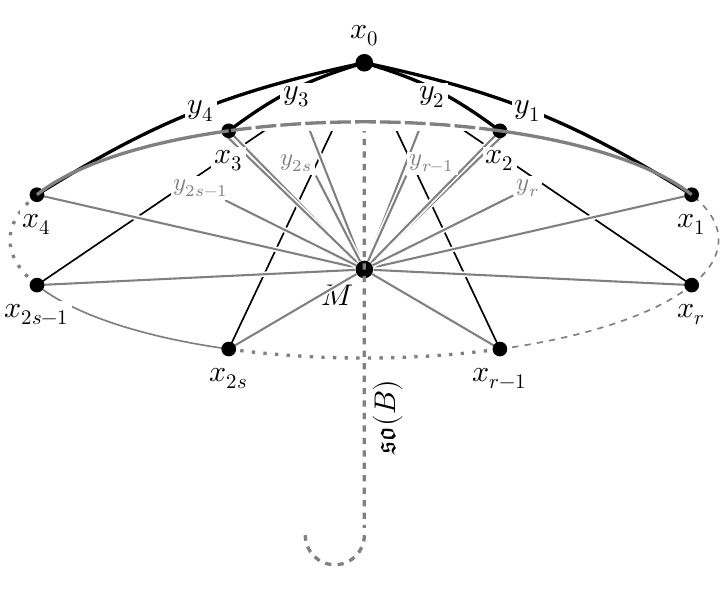} \caption{Bottom view}
        \label{figUm}
    \end{minipage}
\end{figure}

\section{Umbrella Hopf algebras are not IHOEs of primitive parts}
Some basic observations of the Lie algebra $\mathfrak{so}(B)$ help us to illustrate the example. Recall that the symplectic Lie algebra $\mathfrak{sp}_{2s}(\mathbb{k})$ is 
    $$\mathfrak{sp}_{2s}(\mathbb{k}):=\{N\in M_{2s}(\mathbb{k})\mid \Omega N = - N^T \Omega\} = \{N\in M_{2s}(\mathbb{k})\mid \Omega N \Omega= N^T \},$$ 
    where $\Omega=\left(\begin{smallmatrix}
        0 & I_s\\ -I_s & 0
    \end{smallmatrix}\right).$
The symplectic Lie algebra $\mathfrak{sp}_{2s}(\mathbb{k})$ is a simple Lie algebra.
\begin{lemma}\label{lemmalie}
     Let $B\in M_{r}(\mathbb{k})$ be the matrix \eqref{matrixDia} of rank $2s$. If $r=2s$, then $\mathfrak{so}(B)\cong \mathfrak{sp}_{2s}(\mathbb{k})$ as a Lie algebra. Generally, $\mathfrak{sp}_{2s}(\mathbb{k})$ is a Lie subalgebra of $\mathfrak{so}(B)$, and $\mathfrak{so}(B)$ is of $\mathbb{k}$-linear dimension $(r-2s)\times r + 2s^2 + s$. As a consequence, $$\GKdim( \UM(B))=|\mathcal{G}|=(r-2s)\times r + 2s^2 + s +  2r + 1.$$
\end{lemma}
\begin{proof}
    The conclusion is basically derived from the definition and direct calculations. Denote the submatrix of $B$ of full rank $2s$ by $B'$. Then any element $M\in \mathfrak{so}(B)$ is a block matrix $\left(\begin{smallmatrix}
        M^{11} & M^{12}\\ 0 & M^{22}
    \end{smallmatrix}\right)$, where $M^{11}\in \mathfrak{so}(B')$ and $M^{12},M^{22}$ are arbitrary.  
    Note that $B'$ is congruent to $\Omega$, that is,  there is an orthogonal matrix $P$ such that $PB'P^T=\Omega$. Hence
     $\mathfrak{so}(B')$ is isomorphic to the symplectic Lie algebra $\mathfrak{sp}_{2s}(\mathbb{k})$, where the isomorphism is given by $\mathfrak{so}(B') \to  \mathfrak{sp}_{2s}(\mathbb{k}), M \mapsto PMP^T.$ The conclusion follows from $\dim_k  \mathfrak{sp}_{2s}(\mathbb{k}) = 2s^2 +s$ \cite[Example 1.3.5, Table 1]{LiegpIII}.
\end{proof}

To prove the main theorem, we first prove a lemma. 

\begin{lemma}\label{Hopf-alg-gen-by-coradical}
    Suppose $H$ is a connected Hopf algebra of finite GK-dimension such that
     the graded polynomial algebra $\gr H$ is generated by elements of degree $\leqslant m$  as an algebra.
    Then any Hopf subalgebra $K$ of $H$ is generated by $F_m K$ as an algebra.
\end{lemma}
\begin{proof} 
    Since $F_l K=F_l H\cap K$ for any $l\geqslant 0$, $\gr K$ is a Hopf subalgebra of $\gr H$.
    Recall from Theorem \ref{thmzhuang} that both $\gr H$ and $\gr K$ are polynomial algebras. Then by \cite[Corollary 3.2]{zhouRevisited}, $\gr H$ is a graded polynomial extension of $\gr K$. Obviously, the commutative polynomial generators of $\gr K$ have degree no more than $m$. By the proof of Lemma \ref{lemmaPBWexist}, there exists a PBW generating set $\mathcal{G}'$ of $K$ whose principle symbols lie in $\oplus_{i=0}^m(\gr K)_i$. In other words, $\mathcal{G}'\subseteq F_m K$.
\end{proof}
 
We are now ready to prove that the Umbrella Hopf algebra  $\UM(r,2s)$  is not an IHOE of the universal enveloping algebra of its primitive part. Thus $\UM(r,2s)$ is a counter-example to Question \ref{conj1}.

\begin{theorem}\label{theoremCounterex}
Let $H$ be the Hopf algebra $\UM(r,2s)$.
\begin{enumerate}
    \item $H$ has no Hopf subalgebra of codimension $1$ that contains the universal enveloping algebra of its primitive part. 
    \item If $r= 2s$ and $s\ge 2$, then $H$ has no Hopf subalgebra of codimension $1$.
\end{enumerate}
\end{theorem}
\begin{proof}
   We follow the notations in Example \ref{example} and consider the PBW generating set $\mathcal{G}=\{x_{i}\}_{i=0}^r\cup \{y_i\}_{i=1}^r\cup \{X_i\}_{i=1}^d$ of $H$, where $d=(r-2s)\times r + 2s^2 +s$ and $\{X_i\}_{i=1}^d$ is a linear basis of $\mathfrak{so}(B)$ with $B$ being the matrix \eqref{matrixDia}. 
   
   Suppose $K\subseteq H$ is an Hopf subalgebra of codimension $1$ and $K$ contains the primitive part $P(H)$ of $H$.  
   Since $K$ is also a connected Hopf algebra of finite GK-dimension, it follows from Theorem \ref{thmzhuang} and  Proposition \ref{Hopf-alg-gen-by-coradical} that $\gr K\subseteq \gr H$ is a graded polynomial algebra generated by $(\gr K)_1\oplus (\gr K)_2$. As $P(H)= P(K)$, $(\gr K)_1=(\gr H)_1$. Thus, we may write 
   $$\gr K = \mathbb{k}[\overline{x_0},\cdots,\overline{x_r},\overline{X_1},\cdots,\overline{X_d},\overline{y_1'}, \cdots, \overline{y_{r'}'}] \subseteq \mathbb{k}[\overline{x_0},\cdots,\overline{x_r},\overline{X_1},\cdots,\overline{X_d},\overline{y_1}, \cdots, \overline{y_{r}}]= \gr H,$$ where $\{\overline{y_1'}, \cdots, \overline{y_{r'}'}\}$ are generators of $\gr K$ of degree $2$. In fact, we may assume that $$\overline{y_i'}=\sum_{j=1}^r \alpha_{ij} \overline{y_j}, (1 \leqslant i \leqslant r').$$
   Let $y_i' = \sum_{j=1}^r \alpha_{ij} y_j$, it follows from  Lemma \ref{lemmaPBWexist} that the set $\mathcal{G}'=\{x_{i}\}_{i=0}^r\cup \{y_i'\}_{i=1}^{r'}\cup \{X_i\}_{i=1}^d$ satisfying \eqref{eqPolyG} is a PBW generating set of $K$.
   As a consequence, $|\mathcal{G'}|=\GKdim K = \GKdim H -1=|\mathcal{G}|-1$. Therefore $r'=r-1$. Note that $2s\geqslant2$ thence $K$ contains some element $y_i'$ such that $\{\alpha_{i1}, \alpha_{i2}, \cdots, \alpha_{i(2s)}\}$ 
   are not all $0$. Otherwise, $\{y_1', \cdots, y_{r'}'\}$ is contained in the space spanned by  $\{y_{2s+1}, \cdots, y_r\}$ and $r' \leqslant r-2s$.
   Without loss of generality, we may assume $\alpha_{i1}\ne 0$. It follows from $e_{12}\in \mathfrak{so}(B)$ that
   \[[y'_i,e_{12}]= \alpha_{i1} y_2\in K, \textrm{ and so, } y_2 \in K.\]
  Since $e_{21}, e_{1(2i-1)}-e_{(2i)2}, e_{(2i-1)(2i)} \in \mathfrak{so}(B)$ for all $1\leqslant i \leqslant s$, it follows that
    \begin{align*}
       & [y_2,e_{21}] = y_1 \in K,\\
       &[y_1,e_{1(2i-1)}-e_{(2i)2}] = y_{2i-1}\in K, 1\leqslant i \leqslant s,\\
       &[y_{2i-1},e_{(2i-1)(2i)}] = y_{2i}\in K, 1\leqslant i \leqslant s.
    \end{align*}
    Therefore, $\{y_1, \cdots, y_{2s}\} \subseteq K.$ Similarly, since $e_{1(2s+i)}\in \mathfrak{so}(B)$ for any $1\leqslant i\leqslant r-2s$, 
    \[[y_1,e_{1(2s+i)}]=y_{2s+i}\in K, 1\leqslant i\leqslant r-2s.\]

Hence the generating set $\mathcal{G}\subseteq K$, which implies $H \subseteq K$, contradicting to the assumption.
    Hence, the proof of the statement (1) is completed. 
    
    Next, we prove the statement (2). Suppose $r=2s$ and $K$ is a Hopf subalgebra of $H$ of codimension $1$. We may now assume that $P(H) \nsubseteq K$. 
    
    We claim first that $P(K)$ is a Lie subalgebra of $P(H)$ of  codimension $1$.
    
    For convenience let $\GKdim H =n, \dim_{\mathbb{k}} P(H)=p$ and $\dim_{\mathbb{k}} P(K)=t$.
    Then $\GKdim K=n-1, p = d+r+1, n=p+r$.
    The associated graded Hopf algebra $\gr H$ is a graded regular local ring, whose unique graded maximal ideal $\mathfrak{m}_H=\oplus_{i=1}^{\infty}(\gr H)_i$. It is well known that the rank $n$ of the minimal generating set of $\gr H$ coincides with $\dim_{\mathbb{k}}\mathfrak{m}_H/\mathfrak{m}_H^2$. The graded quotient space is actually spanned by $\{\overline{x_{i}} + \mathfrak{m}_H^2 \}_{i=0}^r\cup \{\overline{y_i} + \mathfrak{m}_H^2 \}_{i=1}^r\cup \{\overline{X_i} + \mathfrak{m}_H^2\}_{i=1}^d$, which is a direct sum of a $p$-dimensional space in degree $1$ and a $r$-dimensional one in degree $2$.  Similarly we define $\mathfrak{m}_K$ and $\mathfrak{m}_K/\mathfrak{m}_K^2$ is a codimension-1 subspace of $ \mathfrak{m}_H/\mathfrak{m}_H^2$. Since $P(K) \subsetneq P(H)$, $\mathfrak{m}_K\cap (\gr K)_1 \subsetneq \mathfrak{m}_H \cap (\gr H)_1$. Then the graded space $\mathfrak{m}_K/\mathfrak{m}_K^2$ could only be the sum of a $(p-1)$-dimensional space in degree $1$ and a $r$-dimensional one in degree $2$. The claim is then proved.

   On the other hand, $\{\overline{y_i} + f_i(X)\}_{i=1}^r \subseteq K $, where $f_i(X)\in \mathfrak{m}_H^2\cap (\gr H)_2 = (\gr H)_1^2$. Then $\overline{x_0}\otimes \overline{x_j} - \overline{x_j} \otimes \overline{x_0} \in (\gr K)_1 \wedge (\gr K)_1$ for any $1\leqslant j\leqslant r$. Hence, $x_i \in K$ for all $i$.
    Then $\mathfrak{L}:=\oplus_{i=0}^{2s}\mathbb{k}x_i$ is both a Lie ideal of  $P(H)$ and $P(K)$. It follows that $P(K)/\mathfrak{L}$ is a Lie subalgebra of $P(H)/\mathfrak{L} \cong \mathfrak{so}(B) = \mathfrak{sp}_{2s}(\mathbb{k})$ of codimension $1$. Since the symplectic Lie algebra $\mathfrak{sp}_{2s}(\mathbb{k})$ is a simple Lie algebra, it follows  \cite[Corollary 2.3]{varea1988existence} that $s=1$.
\end{proof}

\begin{corollary}\label{CoroMain}
Let $H$ be the Hopf algebra $\UM(r,2s)$.
    \begin{enumerate}
        \item $H$ is not an IHOE of the universal enveloping algebra of its primitive part.
        \item If $r = 2s$ and $s\geqslant 2$, then $H$ is not a HOE of any Hopf subalgebra.
        \item If $s\geqslant 2$, then $H$ is not an IHOE of $\mathbb{k}$.
    \end{enumerate}
\end{corollary}
\begin{proof}
    Statements (1)-(2) follows directly from the fact that any HOE $K\subseteq H$ is of codimension $1$ \cite[Corollary 2.7]{brown2015connected}. For statement (3), notice that Theorem \ref{theoremCounterex}(2) asks $r=2s$ only for the fact that $\mathfrak{so}(B)=\mathfrak{sp}_{2s}(\mathbb{k})$ is a semisimple Lie algebra. Generally if $r\geqslant 2s$ and $s\geqslant 2$, then a Hopf subalgebra $K$ of codimension 1 still satisfies that $P(K)\subseteq P(H)$ contains $\{x_i\}_{i=0}^r$ is of codimension 1. By induction, if $H$ is an IHOE, then $\mathfrak{so}(B)$ has a complete flag of Lie subalgebras
    \[0 \subsetneq \mathfrak{s}_1 \subsetneq \mathfrak{s}_2 \subsetneq \cdots \subsetneq \mathfrak{s}_{(r-2s) \times r + 2s^2 + s} = \mathfrak{so}(B),\dim_{\mathbb{k}} \mathfrak{s}_i =i .\] Let $B'$ be the submatrix of $B$ of full rank $2s$, then as a space,
    \[\mathfrak{so}(B) = \mathfrak{so}(B') \oplus \mathfrak{g}, \ \mathfrak{g}=\bigoplus_{\substack{1\leqslant i \leqslant r\\ 2s+1 \leqslant j \leqslant r}} \mathbb{k} e_{ij}.\]
    The chain $(\mathfrak{s}_1 \cap \mathfrak{so}(B')) \subseteq \cdots \subseteq (\mathfrak{s}_{(r-2s) \times r + 2s^2 + s} \cap \mathfrak{so}(B'))$ has a sub-chain such that the subset in each step is a proper subset.
   Meanwhile, $\mathfrak{s}_i \cap \mathfrak{so}(B')$ is a Lie subalgebra of $\mathfrak{so}(B)$. Thus $\mathfrak{so}(B')$ has a complete flag, which contradict to the fact that $\mathfrak{so}(B')$ is a simple Lie algebra.
\end{proof}

\begin{remark}
    Among the counterexamples we give to Question \ref{conj1}, the smallest one is $\UM(2,2)$, whose GK-dimension is $8$ by Lemma \ref{lemmalie}. The space of its primitive elements is of linear dimension $6$. Notice that $\mathfrak{sp}_2(\mathbb{k})=\mathfrak{sl}_2(\mathbb{k})$, and $\UM(2,2)$ is an IHOE of $\mathbb{k}$:
    \[\UM(2,2)=\mathbb{k}[x_0][x_1][x_2;\delta_1][y_1;\delta_2][y_2;\delta_3][e_{11}-e_{22};\delta_4][e_{12};\delta_5][e_{21};\delta_6].\]
    The smallest example that is neither an IHOE of primitive part, nor an IHOE of $\mathbb{k}$ is $\UM(4,4)$, which is of GK-dimension 19. The space of the primitive elements is of linear dimension $15$.
\end{remark}

\section{Properties of Umbrella Hopf algebras}
In this section, we describe the structure and give some homological properties of $\UM(r,2s)$.
As given in the construction, the generating set $\mathcal{G}$ of $\UM(r,2s)$ is divided into three parts. In Proposition \ref{PropIterateUL} we will show that $\UM(r,2s)$ is an iterated crossed product of the enveloping algebras of these three parts as a Hopf algebra. 
\subsection{Crossed product of Hopf algebras}
It is well known that the algebra structure of crossed products can be characterized by cleft extensions \cite[Theorem 7.2.2]{montgomery1993hopf}. The corresponding coalgebra and Hopf versions are studied in \cite{AndruExtension}. We first recall the definition of the crossed product as an Hopf algebra following \cite{AndruExtension}.
\begin{definition}
    Let $K$ and $J$ be two Hopf algebras. Suppose there is a $J$-weak action $\gamma$ on $K$ \cite[Definition 2.0]{AndruExtension} (specifically, $K$ is a left $J$-module algebra). A linear map $\sigma:J\otimes J \to K$ is called a cocycle, if
\begin{align*}
    &\sigma(1,a)=\sigma(a,1)=\varepsilon(a)1_K, &\forall a \in J,\\
   & \gamma(a_1,\sigma(b_1,c_1))\sigma(a_2,b_2c_2) = \sigma(a_1,b_1) \sigma(a_2b_2,c), &\forall a,b,c\in J.
\end{align*}
\end{definition}
\begin{definition}  Let $K$ and $J$ be two Hopf algebras.
Suppose there is a $K$-weak coaction $\rho$ on $J$ \cite[Definition 2.10]{AndruExtension} 
$\rho:J\to J\otimes K, a \mapsto \rho(a)^{(1)} \otimes \rho(a)^{(2)}$ (specifically, $J$ is a right $K$-comodule coalgebra). A linear map $\tau:J\to K\otimes K, a \mapsto \tau(a)^{(1)} \otimes \tau(a)^{(2)}$ is called a co-cocycle, if
\begin{align*}
    &\quad \varepsilon_K (\tau(a)^{(1)}) \otimes \tau(a)^{(2)} = \tau(a)^{(1)} \otimes \varepsilon_K(\tau(a)^{(2)})= \varepsilon_J(a) 1_K, & \forall a \in J,\\
    &\quad (\tau(a_1)^{(1)})_1 \tau(\rho(a_2)^{(1)})^{(1)} \otimes (\tau(a_1)^{(1)})_2 \tau(\rho(a_2)^{(1)})^{(2)} \otimes \tau(a_1)^{(2)} \rho(a_2)^{(2)}  &\\
     &=\tau(a_1)^{(1)} \otimes (\tau(a_1)^{(2)})_1 \tau(a_2)^{(1)} \otimes (\tau(a_1)^{(2)})_2 \tau(a_2)^{(2)}, &\forall a \in J.
\end{align*}
\end{definition}

Let $K$ and $J$ be two Hopf algebras with the data $(\gamma,\sigma,\rho,\tau)$ as above. Under certain conditions, $(\gamma,\sigma)$ and $(\rho,\tau)$ give an algebra structure  and  a coalgebra structure on $K\otimes J$ respectively. Furthermore, $K\otimes J$ has a Hopf structure under some compatible conditions as below \cite[Theorem 2.20]{AndruExtension}.
\begin{definition}
    Let $K$ and $J$ be two Hopf algebras. Suppose that there is a $J$-weak action $\gamma$ on $K$ and an $K$-weak coaction $\rho$ on $J$. Let $\sigma:J\otimes J\to K$ be a cocycle and $\tau:J\to K\otimes K$ be a co-cocycle. If the vector space $K\otimes J$, with the following structures, is a Hopf algebra, then it is called the crossed product  of $K$ and $J$ with the data $(\gamma,\sigma,\rho,\tau)$, denoted by $K^{\tau}\#_{\sigma}J$.
    \begin{enumerate}
        \item (unit) $1_{K^{\tau}\#_{\sigma}J}=1_K\# 1_J$.
        \item (product) $(h\# a)(g\# b):= h \gamma(a_1,g)\sigma(a_2,b_1)\# a_3b_2$.
    \item (counit) $\varepsilon_{K^{\tau}\#_{\sigma}J} =\varepsilon_K\#\varepsilon_J $.
    \item (coproduct) $\Delta(h\#a) =( h_1 \tau(a_1)^{(1)} \# \rho(a_2)^{(1)}) \otimes (h_2 \tau(a_1)^{(2)} \rho(a_2)^{(2)} \# a_3 ) $.
    \item (antipode) $S(h\#a) = (\gamma(S(\rho(a)^{(1)})_2,S(\rho(a)^{(2)}))\# S(\rho(a)^{(1)})_1)(S h \#1)$.
    \end{enumerate} 
   
    The cocycle $\sigma$ is often omitted in the notation if $\sigma(a,b)=\varepsilon(a)\varepsilon(b)1_K$ for any $a,b\in J$. Similarly, the co-cocycle $\tau$ is omitted if $\tau(a) = \varepsilon(a)1_K \otimes 1_K$ for any $a\in J$.
\end{definition}

Instead of checking the compatibility conditions, which are tedious, we check the equivalent characterization of cleft extension as a Hopf algebra \cite[Definition 3.2.13]{AndruExtension}. As is summarized in \cite[Proposition 3.1.12]{AndruExtNote}, any cleft extension $H$ of $K$ by $J$ as a Hopf algebra is isomorphic to a crossed product $K^{\tau}\#_{\sigma}J$.
\begin{definition}\label{defCleft}
    A Hopf algebra $H$ is a cleft extension of the Hopf algebra $K$ by the Hopf algebra $J$ if the following conditions hold.
    \begin{enumerate}
        \item There is an exact sequence of Hopf algebras  \cite[Definition 1.2.0]{AndruExtension}
        $$0\to K \xrightarrow{i} H \xrightarrow{\pi} J \to 0.$$
        \item There is a convolution invertible right $J$-comodule map $\chi_H:J\to H$ such that $\chi_H(1_J) = 1_H$.
        \item There is a convolution invertible left $K$-module map $\xi_H:H \to K$ such that $\varepsilon_K \xi_H = \varepsilon_H$.
        \item $\xi_H \chi_H = \varepsilon_J 1_K.$
    \end{enumerate}
\end{definition}

\begin{lemma}\label{LemmaCleft} The conditions (1) and (2) in
   Definition \ref{defCleft} with $\varepsilon_H\chi_H = \chi_J$ are sufficient to give a cleft extension.
\end{lemma}
\begin{proof} See \cite[Lemma 3.1.14]{AndruExtNote}.
\end{proof}

The Hopf Ore extension of a connected Hopf algebra of invariant type defined in \cite{brown2015connected} is a special case of cleft extension as a Hopf algebra. Note that a Hopf Ore extension $K\subseteq H$ is of invariant type if and only if $K^+H$ is an left and right ideal of $H$. By \cite[Theorem 1.9]{brown2015connected} such a Hopf algebra $H$ could always be realized as $K[z;\partial]$ for some $z\in H$ and $\partial \in \Der(K)$. In the following lemma, we use  $\overline{z}$ to denote $z + K^+H \in \overline{H} = H/K^+H$.
 \begin{lemma}
     Suppose $K$ is a connected Hopf algebra and $H = K[z;\partial]$ is a Hopf Ore extension of invariant type. Then $H$ is also connected and $H\cong K^{\tau} \# \mathbb{k}[\overline{z}]$ is a crossed product as a Hopf algebra.
 \end{lemma}
 \begin{proof}
     First of all, we may assume that $z\in H^+$, otherwise replace it with $z-\varepsilon(z)$. Since $H$ is a free $K$-module with basis $\{z^i\}_{i=0}^{\infty}$, $\overline{H}:=H/K^+H = \oplus_{i=0}^{\infty}\mathbb{k}\overline{z}^i$ as a space. As $K^+H$ is an ideal, $\overline{H}$ is a Hopf algebra and $\overline{H} = \mathbb{k}[\overline{z}]$ as an algebra. The freeness also indicates that ${}_KH$ is faithfully flat. In other words, $K$ is a quantum homogeneous space of $H$ \cite[Proposition 3]{takeuchi1979relative} and $K = {}^{co \overline{H}}H = H^{co \overline{H}}$. The exact sequence $0\to K \xrightarrow{i} H \xrightarrow{\pi} \bar{H} \to 0$ of Hopf algebras
     follows directly \cite[Proposition 1.2.3]{AndruExtension}. By \cite[Theorem 1.3]{brown2015connected}, $\delta z \in K\otimes K$. Being more precisely, $\delta z \in K^+ \otimes K^+$ as $z\in H^+$. Thus $\Delta \overline{z} = \overline{z} \otimes 1 + 1\otimes \overline{z}$, and the following maps from $\overline{H}$ to $H$ are $\overline{H}$-comodule homomorphisms.
     \begin{equation*}
         \chi_H(\overline{z}^i) = z^i,\ \chi'_H(\overline{z}^i) = (-z)^i.
     \end{equation*}
     It is direct to check that $\chi_H(1_{\overline{H}}) = 1_H$ and $\varepsilon_H \chi_H = \varepsilon_{\overline{H}}$. And the convolution of $\chi_H$ and $\chi'_H$ is the unit in $\Hom(\overline{H},H)$.
     \begin{align*}
         \chi_H\star \chi'_H(\overline{z}^i) = \sum_{j=0}^i \C^j_i \chi_H(\overline{z}^j)\chi'_H(\overline{z}^{i-j}) = (\sum_{j=0}^i \C^j_i (-1)^{i-j}) z^i = 0, \forall i\geqslant 1.
     \end{align*}
     By Lemma \ref{LemmaCleft}, $H$ is a cleft extension of $K$ by $\overline{H}$, so there exist cocycle $\sigma$ and co-cocycle $\tau$ such that $H\cong K^{\tau}\#_{\sigma} \mathbb{k}[\overline{z}]$. By \cite[Theorem 3.2.1]{AndruExtension}, 
     \[\sigma(u,v) = \chi_H(u_1) \chi_H(v_1) \chi'_H(u_2v_2), \forall u,v\in \overline{H}. \]
     Then by the map $\chi_H$ and $\chi'_H$ given above, $\sigma(u,v) = \varepsilon(u) \varepsilon(v) 1_K$.
 \end{proof}
 
Note that $\mathbb{k}[\overline{z}]$ in the above lemma is the enveloping algebra of $\mathbb{k}\overline{z}$ as a Hopf algebra. So if $H$ is an IHOE over $U(P(H))$ with all steps being invariant, then $H$ could be realized as an ``iterated crossed product of enveloping algebras". Meanwhile, the explicit examples \cite{wang2015connected} of IHOEs over $U(P(H))$ of variant type could also be realized as crossed products of enveloping algebras by other means. We name a simplest one in the following example \cite[Example 4.2]{wang2015connected}.
\begin{example}
    Let $H$ be the algebra generated by $x,y,z$ subject to the following relations.
    \begin{align*}
        [x,y] &= y,\\
        [z,y] &= 0,\\
        [z,x] &= -z + \lambda y, \lambda \in \mathbb{k}.
    \end{align*}
    $H$ becomes a Hopf algebra with $x,y\in P(H)$ and $\delta z = x\otimes y - y\otimes x$.

    Then $K = \mathbb{k}[y]$ is a Hopf subalgebra of $H$ and $K^+H = HK^+$, which means the following sequence is an exact sequence of Hopf algebras:
    \begin{equation}\label{eqExactExampleWZZ}
        0\to K \to H \to \overline{H}=H/K^+H \to 0.
    \end{equation}
    To define a comodule map from $\overline{H}$ to $H$, we adjust the generator $z$ to be $z' = z + xy$. Obviously $\overline{z} = \overline{z'}$ and $\{\overline{x},\overline{z'}\}$ is a PBW generating set of $\overline{H}$. The following maps are convolution inverse of each other, and $\chi_H$ is a $\overline{H}$-comodule homomorphism.
    \begin{equation*}
        \chi_H(\overline{x}^i \overline{z'}^j) = x^i z'^j,\ \chi'_H(\overline{x}^i \overline{z'}^j) = (-1)^{i+j} z'^j x^i.
    \end{equation*}
    By definition, $\chi_H$ satisfies Lemma \ref{LemmaCleft}, and thus \eqref{eqExactExampleWZZ} is a cleft extension of Hopf algebras. Notice that in $\overline{H}$, the generator $\overline{z'}$ becomes a primitive element. So, as a Hopf algebra, $\overline{H}$ is actually the enveloping algebra of the non-trivial $2$-dimensional Lie algebra. The Hopf algebra $H$ is a crossed product of $2$ enveloping algebras.
\end{example}

\subsection{Umbrella algebras are iterated crossed products}
 Let us now focus on the Hopf algebra $H=\UM(r,2s)$. Notice that the subalgebra $K$ generated by $\{y_i\}_{i=1}^r\cup \{x_i\}_{i=0}^r$ and $K'$ generated by $\{x_i\}_{i=0}^r$ are both Hopf subalgebras. In the following proposition, we use the notation $\overline{y_i}$ (resp. $\overline{X_i}$) to denote $y_i+(K')^+K$ (resp. $X_i + K^+\UM(r,2s)$) in the quotient algebras.
\begin{proposition}\label{PropIterateUL}
    As a Hopf algebra, $H=\UM(r,2s)$ is isomorphic to an iterated crossed product of the enveloping algebras:
    \[H\cong (U(\mathfrak{L})^{\tau}\#_{\sigma}U(Y))\#U(\mathfrak{so}(B)),\]
    where $\mathfrak{L}$ is the Lie algebra spanned by $\{x_i\}_{i=0}^r$ with the Lie bracket given by \eqref{relation1} and \eqref{relationxx}, and $Y$ is the trivial Lie algebra spanned by $\{\overline{y_i}\}_{i=1}^r$. The Hopf algebra $U(Y)$ acts on $U(\mathfrak{L})$ trivially, and the cocycles are given by
\begin{align*}
    \sigma(\overline{y_i},\overline{y_j})=\left\{
    \begin{aligned}
        &-\frac{1}{3} x_0^3,\ i=2k, j=2k-1\\
         &0,\ otherwise.
    \end{aligned}
             \right.
\end{align*}
\[\tau(\overline{y_i}) = x_0 \otimes x_i - x_i \otimes x_0.\]
   The action of $U(\mathfrak{so}(B))$ on $U(\mathfrak{L})^{\tau}\#_{\sigma}U(Y)$ is defined as follows.
   \[\overline{M} \rightharpoonup (x_i\#\overline{y_j}) = \sum_{k=1}^r (-M_{ik}x_k\#\overline{y_j} - x_i\#M_{jk}\overline{y_k}). \]
\end{proposition}

\begin{proof}
    Let $K$ be the subalgebra of $H$ generated by $\mathfrak{L}\cup Y$, and $K'$ be the one generated by $\mathfrak{L}$. By the definition of the coproduct and antipode given in Example \ref{example}, $K$ and $K'$ are both Hopf subalgebras of $H$. We claim that $K'\subseteq K\subseteq H$ are cleft extensions as Hopf algebras. By \cite[Proposition 3]{takeuchi1979relative}, both $K$ and $K'$ are quantum homogeneous space of $H$. Moreover, relations \eqref{relation1}-\eqref{relationlast} tells that $[H,K^+]\subseteq K^+$, $[K,(K')^+]\subseteq (K')^+$, which means $\overline{H}=H/K^+H$ and $\overline{K}=K/(K')^+K$ are also Hopf algebras. By \cite[Proposition 1.2.3]{AndruExtension}, $0\to K'\to K \to \overline{K}\to 0$ is an exact sequence of Hopf algebras, and so is $0\to K \to H \to \overline{H}\to 0$.
    
    Obviously,  $\overline{y_i} \in \overline{K}$ becomes a primitive element, and $[\overline{y_i},\overline{y_j}]\in \mathbb{k} \overline{x_0^3}=0$. 
    Then $\overline{K}=U(Y)$ with the trivial Lie bracket, and $\{\overline{y_i}\}_{i=1}^r$ is a PBW generating set of $\overline{K}$.
    The following $\overline{K}$-comodule maps  from $\overline{K}$ to $K$ are defined via the action the PBW basis.
    \begin{align*}
        \chi_K &:\overline{y_1}^{\alpha_1}\overline{y_2}^{\alpha_2}\cdots \overline{y_r}^{\alpha_r} \mapsto y_1^{\alpha_1}y_2^{\alpha_2} \cdots y_r^{\alpha_r},\\
        \chi'_K&: \overline{y_1}^{\alpha_1}\overline{y_2}^{\alpha_2}\cdots \overline{y_r}^{\alpha_r} \mapsto (-1)^{\sum \alpha_i} y_r^{\alpha_r}y_{r-1}^{\alpha_{r-1}} \cdots y_1^{\alpha_1}.
    \end{align*}
    Consider their convolution acting on an arbitrary basis element whose last letter is $y_t$:
    \begin{align*}
        ~&(\chi_K \star \chi'_K)(\overline{y_1}^{\alpha_1}\cdots \overline{y_t}^{\alpha_t}) \\
         =& \sum_{0\leqslant \beta_i\leqslant \alpha_i} 
(\Pi_{i\leqslant t} \C_{\alpha_i}^{\beta_i}) \chi_K(\overline{y_1}^{\beta_1}\cdots \overline{y_t}^{\beta_t}) \chi_K{}'(\overline{y_1}^{\alpha_1-\beta_1}\cdots \overline{y_t}^{\alpha_t-\beta_t}))\\
=&\sum_{0\leqslant \beta_i\leqslant \alpha_i} 
(\Pi_{i\leqslant t} \C_{\alpha_i}^{\beta_i}) y_1^{\beta_1}\cdots y_t^{\beta_t}\cdot (-1)^{\sum \alpha_i - \beta_i} y_t^{\alpha_t-\beta_t}\cdots y_1^{\alpha_t-\beta_t}\\
 = &\sum_{0\leqslant \beta_i \leqslant \alpha_i, i<t} (\sum_{\beta_t=0}^{\alpha_t} (-1)^{\alpha_t-\beta_t} \C_{\alpha_t}^{\beta_t}) \cdot (\Pi_{i<t} \C_{\alpha_i}^{\beta_i})  y_1^{\beta_1}\cdots y_t^{\alpha_t}\cdots y_1^{\alpha_1-\beta_1}\\
 =&0 = \varepsilon(\overline{y_1}^{\alpha_1}\cdots \overline{y_t}^{\alpha_t}), \forall \{\alpha_i\} \text{ such that } \alpha_i \text{ are not all 0}.
    \end{align*}
    So $\chi_K$ is a convolution invertible comodule map. By \cite[Theorem 3.2.1]{AndruExtension}, $K=K'\#_{\sigma}\overline{K}$ as an algebra, where
    $\sigma(u,v)=\chi_K(u_1)\chi_K(v_1)\chi'_K(u_2v_2)$ for any element $u,v\in U(Y)$, and $\overline{K}$ acts on $K'$ by \[\overline{y_i}\rightharpoonup x_j = \chi_K(\overline{y_i})_1 x_j \chi'_K(\overline{y_i})_2 = y_i x_j - x_j y_i =0.\]
    Note that $\chi_K(\overline{y_i}\overline{y_j}) = \chi_K(\overline{y_j}\overline{y_i})$, while $\chi_K(\overline{y_i})\chi_K(\overline{y_j})\ne \chi_K(\overline{y_j}) \chi_K(\overline{y_i})$ in general, which makes $\sigma$ a non-trivial twist. For example,
    \begin{align*}
    \sigma(\overline{y_{2k}},\overline{y_{2k-1}}) =& \chi_K(\overline{y_{2k}}) \chi_K(\overline{y_{2k-1}}) + \chi_K(\overline{y_{2k}}) \chi'_K(\overline{y_{2k-1}}) +\\ &\chi_K(\overline{y_{2k-1}}) \chi'_K(\overline{y_{2k}}) + \chi'_K(\overline{y_{2k}}\overline{y_{2k-1}}) \\
    =&y_{2k} y_{2k-1} - y_{2k}y_{2k-1} - y_{2k-1}y_{2k} + y_{2k}y_{2k-1}=-\frac{1}{3} x_0^3.\\
    \sigma(\overline{y_{2k-1}},\overline{y_{2k}})=&y_{2k-1}y_{2k} - y_{2k-1}y_{2k} - y_{2k}y_{2k-1} + y_{2k}y_{2k-1} = 0.
        \end{align*}

For the coproduct of $K$ we check that there is an invertible module map from $K$ to $K'$. By the PBW basis $K$ is a free left $K'$-module with basis $\{y_1^{\alpha_1}\dots y_r^{\alpha_r}\mid \alpha_i \in \mathbb{N}\}$. So as a $K'$-module, $K = K' \oplus K'\chi_K(\overline{K}^+)$. Thus there is a $K'$-module projection $\xi_K$ from $K$ to $K'$. Let $\xi'_K$ be $S\xi_K$. Denote the product on $K$ by $m_K$,
\begin{align*}
    &(\xi_K \star \xi'_K)(\Pi_{i=0}^r x_i^{\alpha_i} \Pi_{i=1}^r y_i^{\beta_i})\\
    =& m_K (\xi_K \otimes \xi'_K)(\Pi_{i=0}^r (x_i\otimes 1 + 1\otimes x_i)^{\alpha_i} \Pi_{i=1}^r(x_0 \otimes x_i - x_i \otimes x_0)^{\beta_i})\\
    =&m_K(\id \otimes S)(\Pi_{i=0}^r (x_i\otimes 1 + 1\otimes x_i)^{\alpha_i} \Pi_{i=1}^r(x_0 \otimes x_i - x_i \otimes x_0)^{\beta_i})
\end{align*}
Note that if $a,b \in K\otimes K$ and $a,b\in \ker(m\circ(id \otimes S))$, then $ab \in \ker(m\circ(id \otimes S))$. Obviously $x_i\otimes 1 + 1\otimes x_i$ and $(x_0\otimes x_i - x_i\otimes x_0)$ are in the kernel of $m\circ(id\otimes S)$. Thus 
$\xi_K \star \xi'_K = \varepsilon$, and so it $\xi'_K\star \xi_K$. By definition $\xi_K$ and $\chi_K$ satisfies \ref{defCleft}, so that $K$ is a (Hopf) cleft extension of the Hopf algebra $K'$ by the Hopf algebra $\overline{K}$. In other words, $K \simeq K'^{\tau}\#_{\sigma}\overline{K}$, where $\tau:\overline{K} \to K'\otimes K'$ is given by $\xi_K$ \cite[Proposition 3.2.9]{AndruExtension} as
\[\tau(u+K'{}^+K) = \Delta_{K'}(\xi_K^{-1}(u_1)) (\xi_K(u_2)\otimes \xi_K(u_3)) .\]
Then $\tau(\overline{y_i}) = x_0\otimes x_i - x_i \otimes x_0$.

    As for the quotient $\overline{H}$, it is isomorphic to the enveloping algebra of $\mathfrak{so}(B)$. We fix an ordered set of $\mathbb{k}$-basis on the Lie algebra and denote the element in $\overline{H}$ by $\overline{X}_1<\overline{X}_2<\dots<\overline{X}_n$. Then similarly a convolution invertible map can be defined from $\overline{H}$ to $H$:
    \begin{align*}
        \chi_H &:\overline{X_1}^{\alpha_1}\overline{X_2}^{\alpha_2}\cdots \overline{X_n}^{\alpha_n} \mapsto X_1^{\alpha_1}X_2^{\alpha_2} \cdots X_n^{\alpha_n},\\
        (\chi_H)^{-1}&: \overline{X_1}^{\alpha_1}\overline{X_2}^{\alpha_2}\cdots \overline{X_n}^{\alpha_n} \mapsto (-1)^{\sum \alpha_i} X_n^{\alpha_n}X_{n-1}^{\alpha_{n-1}} \cdots X_1^{\alpha_1}.
    \end{align*}
    Then $H=K\#_{\nu}\overline{H}$ as an algebra for some $\nu$ and the action of $\overline{H}$ on $K$ is 
    $$\overline{M}\rightharpoonup(x_i\#\overline{y_j})=[M,x_iy_j]=[M,x_i]\#\overline{y_j} + x_i\#\overline{[M,y_j]}.$$ 
    However it is easy to observe that $\chi_H|_{\mathfrak{so}(B)}$ is a Lie algebra injection from $\mathfrak{so}(B)$ to $P(H)$, so by the universal property of $U(\mathfrak{so}(B))$, $\chi_H(\overline{X_{i_1}}\overline{X_{i_2}}\cdots\overline{X_{i_t}})= X_{i_1}X_{i_2}\cdots X_{i_t}$ for any  $(i_1,\dots,i_t)$. Similarly $(\chi_H)^{-1} = S\circ \chi_H$. As a consequence, $\nu(u,v) = \varepsilon(u)\varepsilon(v)1_H$ for any $u,v\in \overline{H}$.

    On the other hand, a $K$-module projection $\xi_H$ from $H$ to $K$ is given by the module decomposition ${}_K H ={}_K K \oplus {}_K(K \chi_H(\overline{H}^+)) $ similarly to $\xi_K$. It it also invertible, and satisfies \ref{defCleft}. The twist induced by $\xi_H$ is trivial by direct calculation. So overall, $H \simeq K \# U(\mathfrak{so}(B))$ as a Hopf algebra.
\end{proof}
Naturally we ask the following question.
\begin{que}
    Let $H$ be an arbitrary connected Hopf algebra. Is there a finite number of Lie algebras $\mathfrak{g}_i$ such that $H$ is an iterated crossed product of the cocommutative Hopf algebras $\{U(\mathfrak{g}_i)\}$?
\end{que}
\subsection{Calabi-Yau property of the umbrella algebras}
By Theorem \ref{thmzhuang} any connected Hopf algebra $H$ of finite GK-dimension is skew Calabi-Yau, that is, $H$ is homological smooth (of dimension $d$), and there exists $\sigma\in \Aut(H)$ such that as $H$-$H$-bimodules,
\begin{align*}
    \Ext_{H^e}^i(H,H^e)\cong \left\{
    \begin{array}{rl}
        0 ,& i< d, \\
        {}^1 H^{\sigma}, & i=d.
    \end{array} \right.
\end{align*}
The automorphism $\sigma$ is called a Nakayama automorphism of $H$, and $H$ is called Calabi-Yau if $\sigma$ is an inner automorphism. The following proposition gives the Nakayama automorphism of $\UM(r,2s)$.
\begin{proposition}\label{PropCY}
    The Hopf algebra $H=\UM(r,2s)$ is a skew Calabi-Yau algebra with Nakayama automorphism
    \[\sigma(x_i)=x_i, \sigma(y_i)=y_i, \]
    \[\sigma(M) = M + (2-2s)\tr(M),\forall M\in \mathfrak{so}(B).\]
If $r=2s$ then $H$ is a Calabi-Yau algebra.
\end{proposition}
\begin{proof}
    By Theorem \ref{thmzhuang}, the associated graded algebra $\gr H$ of $H(=\UM(B)=\UM(r,2s))$ with respect to the coradical filtration is a commutative polynomial ring. As the elements $\overline{y_i}$ lie in $(\gr H)_2$ and the principle symbols of other elements in $\mathcal{G}$ lie in $(\gr H)_1$, relations \eqref{relation1}-\eqref{relationlast} ensure that $\gr H=\mathbb{k}[\overline{\mathcal{G}}]$. Then $H$ fits in the framework of \cite[Theorem 0.1]{wu2021nakayama}, that is, $H$ is a filtered deformation of the Poisson polynomial algebra $(\gr H,\{-, -\})$, where
    \[\{a+F_{m-1}H,b+F_{n-1}H\}=[a,b]+F_{m+n-2}H, \forall a \in F_mH, b\in F_nH.\] 
    It follows that the Nakayama automorphism $\sigma$ of $H$ satisfies that:
    \[\sigma(a)-a = \phi_{\eta}(a),\forall a \in P(H),\]
    \[\sigma(y_i)-y_i - \phi_{\eta}(y_i) = \alpha_i \in F_0 H,\forall 1\leqslant i\leqslant r\]
    where $\phi_{\eta}:  H \to \mathbb{k}, a \mapsto \textrm{Hdet} \{\overline{a},-\}$,  and $\textrm{Hdet}\{\overline{a},-\}$  is the homological determinant of the Hamilton derivation $\{\overline{a},-\}$ on the Poisson polynomial algebra.
    In our cases, for any $a\in \mathcal{G}$, $\phi_{\eta}(a)$ is exactly the trace of $\{\overline{a},-\}$ acting on $\overline{\mathcal{G}}$, see \cite[Corollary 0.5, Proposition 0.7]{wu2021nakayama} for details. By direct calculation,
    \begin{align*}
    \phi_{\eta}(x_i) &= 0,& \forall 0\leqslant i \leqslant r\\
    \phi_{\eta}(y_i) &= 0, &\forall 1\leqslant i \leqslant r\\
        \phi_{\eta}(M) &= 2\tr (M) + \tr(\ad_{\mathfrak{so}(B)} M), &\forall M\in \mathfrak{so}(B).
    \end{align*}
    Let $B'$ be the submatrix of $B$ of full rank $2s$, then $\mathfrak{so}(B')$ is a simple Lie algebra. Recall from Lemma \ref{lemmalie} that  $\mathfrak{so}(B)$ consists of the block matrices $\left(\begin{smallmatrix}
        M^{11} & M^{12}\\ 0 & M^{22}
    \end{smallmatrix}\right)$, where $M^{11}\in \mathfrak{so}(B')$ and $M^{12},M^{22}$ are arbitrary matrices. Pick a linear basis 
    \[b:=b(\mathfrak{so}(B'))\cup\{e_{ij}\}_{\substack{1\leqslant i \leqslant r,\\ 2s+1\leqslant j \leqslant r}}.\]
    For a matrix $M=\sum_{X\in b} \alpha_X X$, we denote $X^*(M)=\alpha_X$. Then by direct calculation, 
    \begin{align*}
        \tr(\ad_{\mathfrak{so}(B)} M) =& \sum_{X\in b} X^*([M,X])\\
        =& \tr(\ad_{\mathfrak{so}(B')} M^{11}) + \sum_{\substack{1\leqslant i \leqslant r,\\ 2s+1\leqslant j \leqslant r}}(e_{ij})^*([M,e_{ij}])\\
        =& 0 + \sum_{\substack{1\leqslant i \leqslant r,\\ 2s+1\leqslant j \leqslant r}} (M_{ii}-M_{jj})\\
        =&(r-2s) \tr(M) - r \tr(M^{22}).
    \end{align*}
    However by definition, $M^{11}B'$ is a symmetric matrix. Then $M^{11}=(M^{11}B')(B')^{-1}$ is a product of a symmetric matrix and an anti-symmetric one, whose trace must be 0. As a consequence, $\tr(M)=\tr(M^{22})$, and $\phi_{\eta}(M)=(2-2s)\tr(M)$.

      On the other hand, $\sigma$ is an algebra automorphism, so it has to satisfy relations \eqref{relation1}-\eqref{relationlast}. Especially, for any element $M\in \mathfrak{so}(B)$ and $1\leqslant i \leqslant r$,
    \[\sum_j M_{ij}\sigma(y_{j})=\sigma([M,y_i])=[\sigma(M),\sigma(y_i)]=[M+2\tr(M),y_i + \alpha_i]=\sum_j M_{ij}y_{j}. \]
    Let $\hat{B}=B + \sum_{i=2s+1}^r e_{ii} $, which lies in $ \mathfrak{so}(B)$. Since $\hat{B}$ is of full rank, $\sigma(y_i)-y_i$ is $0$ for any $i$.
    
    If $r=2s$, then $\tr(M)=\tr(M^{11})=0.$ The Nakayama automorphism $\sigma$ acts on the generators trivially, which makes $H$ a Calabi-Yau algebra.
\end{proof}

\bibliographystyle{amsalpha}
\bibliography{bib.bib}

\end{document}